\documentclass[11pt]{article}

\usepackage[paper=letterpaper, margin=1in]{geometry}

\usepackage{ifthen}

\ifthenelse{\isundefined{\GenerateShortVersion}}
{
\def\LongVersion{}
\def\LongVersionEnd{}
\long\def\ShortVersion#1\ShortVersionEnd{}
}
{
\def\ShortVersion{}
\def\ShortVersionEnd{}
\long\def\LongVersion#1\LongVersionEnd{}
}

\newcommand{\Ignore}[1]{\ignorespaces}

\usepackage{amsmath, amssymb}

\usepackage{amsthm}

\usepackage{ifpdf}

\usepackage{authblk}

\usepackage{bm}

\ifpdf
\usepackage[pdftex]{graphicx}
\else
\usepackage[dvips]{graphicx}
\fi

\usepackage[%
  ocgcolorlinks,%
  linkcolor=blue,%
  filecolor=blue,%
  citecolor=blue,%
  urlcolor=blue]{hyperref}

\LongVersion 

\LongVersionEnd 

\ShortVersion 
\ShortVersionEnd 

\ShortVersion 
\ShortVersionEnd 

\ShortVersion 
\renewcommand{\paragraph}[1]{\par\noindent\textbf{#1}}
\ShortVersionEnd 

\newtheorem{theorem}{Theorem}[section]
\newtheorem{lemma}[theorem]{Lemma}

\newtheorem{proposition}[theorem]{Proposition}

\theoremstyle{definition}
\newtheorem{definition}[theorem]{Definition}
\theoremstyle{plain}

\newtheorem*{theorem*}{Theorem}
\newtheorem*{conjecture*}{Conjecture}

\newcommand{\bbE}{\mathbb{E}}

\newcommand{\bbP}{\mathbb{P}}

\newcommand{\bbZ}{\mathbb{Z}}
\newcommand{\bbR}{\mathbb{R}}

\newcommand{\wh}{\widehat}

\newcommand{\cA}{\mathcal A}
\newcommand{\cP}{\mathcal{P}}

\newcommand{\cB}{\mathcal B}

\newcommand{\cW}{\mathcal W}
\newcommand{\cS}{\mathcal S}
\newcommand{\cF}{\mathcal F}
\newcommand{\cR}{\mathcal R}
\newcommand{\cD}{\mathcal D}

\newcommand{\rmd}{{\rm d}}
\newcommand{\rme}{{\rm e}}

\newcommand{\rmB}{{\rm B}}

\newcommand{\wt}{\widetilde}

\begin{document}

\LongVersion 
\title{Exploring an Infinite Space with Finite Memory Scouts}
\LongVersionEnd 
\ShortVersion 
\title{Exploring an Infinite Space with Finite Memory Scouts \\ (Extended
Abstract)}
\ShortVersionEnd 

\author{Lihi Cohen}
\affil{Technion, Israel.
Email: \texttt{colihi@gmail.com}}

\author{Yuval Emek}
\affil{Technion, Israel.
Email: \texttt{yemek@ie.technion.ac.il}}

\author{Oren Louidor}
\affil{Technion, Israel.
Email: \texttt{oren.louidor@gmail.com}}

\author{Jara Uitto}
\affil{Comerge AG, Switzerland.
Email: \texttt{jara.uitto@comerge.net}}

\date{}

\begin{titlepage}
\maketitle
\thispagestyle{empty}

\begin{abstract}
Consider a small number of \emph{scouts} exploring the infinite $d$-dimensional
grid with the aim of hitting a hidden target point.
Each scout is controlled by a probabilistic finite automaton that determines
its movement (to a neighboring grid point) based on its current state.
The scouts, that operate under a fully synchronous schedule, communicate
with each other (in a way that affects their respective states) when they
share the same grid point and operate independently otherwise.
Our main research question is:
How many scouts are required to guarantee that the target admits a \emph{finite
mean hitting time}?
Recently, it was shown that
$d + 1$
is an upper bound on the answer to this question for any dimension
$d \geq 1$
and the main contribution of this paper comes in the form of proving that this
bound is tight for
$d \in \{ 1, 2 \}$.
\end{abstract}

\ShortVersion 
\vskip 1cm
\begin{center}
\Large{A full version that contains all missing proofs along with some
additional material is attached to this extended abstract.}
\end{center}
\ShortVersionEnd 

\end{titlepage}

\section{Introduction}

\paragraph{Model and Contribution.}
Consider $c$ \emph{scouts}
$u^{1}, \dots, u^{c}$
that explore the
$d$-dimensional infinite grid $\bbZ^{d}$.
Each scout is controlled by a probabilistic finite automaton with (finite)
\emph{state set} $\cS$ so that at any given time
$n \in \bbZ_{\geq 0}$,
the global system \emph{configuration} is characterized by the scouts'
location vector
$\mathbf{X}_{n} \in \bbZ^{d \times c}$
and state vector
$\mathbf{Q}_{n} \in \cS^{c}$.\footnote{%
Throughout, we use the boldfaced notation $\mathbf{v}$ to denote the vector
$\mathbf{v} = (v^{1}, \dots, v^{c})$.
}

The scouts \emph{sense} their respective grid points through the following
mechanism:
At time $n \in \bbZ_{\geq 0}$, each scout $u^{i}$ observes its local
(binary) \emph{environment} vector
$E_{n}^{i} \in \{ 0, 1 \}^\cS$
defined so that
$E_{n}^{i}(q) = 1$
if and only if there exists some
$j \neq i$
such that
$X_{n}^{j} = X_{n}^{i}$
and
$Q_{n}^{j} = q$.
In other words, scout $u^{i}$ can sense the current state $q \in \cS$ of another scout $u^j$ if they reside in the same grid point.
Notice that the environment $E_{n}^{i}$ is fully determined by the
configuration
$(\mathbf{X}_{n}, \mathbf{Q}_{n})$.

The scouts' actions are controlled by a (probability) \emph{transition
function}\footnote{%
A function
$\Pi: \cA \times \cB \mapsto [0,1]$,
where $\cA$ and $\cB$ are countable non-empty sets,
is called a (probability) transition function or transition kernel if
$\sum_{b \in \cB} \Pi(a, b) = 1$
for all
$a \in \cA$.
}
\[
\Pi
:
\big( \cS \times \{ 0, 1 \}^{\cS} \big)
\times
\big( \cS \times \{ -1, 0, +1 \}^{d} \big)
\rightarrow
[0, 1]
\]
defined so that
$\Pi \left( (q, e), (q', \xi) \right)$
is the probability that at time
$n + 1$,
scout $u^{i}$ resides in state $Q^i_{n+1} = q'$ in grid point
$X^i_{n+1} = x + \xi$
given that
$Q_{n}^{i} = q$,
$E_{n}^{i} = e$,
and
$X_{n}^{i} = x$.
The application of the transition function $\Pi$ for scout $u^{i}$ at time $n$
is independent of its application for scout $u^{j}$ at time $n'$ for any
$j \neq i$
or
$n \neq n'$.\footnote{%
The choice of introducing a single transition function $\Pi$ for all
scouts, rather than a separated transition function $\Pi^{i}$ for each scout
$u^{i}$, does not result in a loss of generality, as the space set $\cS$ can be
partitioned into pairwise disjoint subsets
$\cS = \cS^{1} \cup \cdots \cup \cS^{c}$,
ensuring that
$Q_{n}^{i} \in \cS^{i}$
for every
$1 \leq i \leq c$
and
$n \in \bbZ_{\geq 0}$.
}

To complete the model specification, each scout $u^{i}$ is associated with an
\emph{initial state}
$q_{0}^{i} \in \cS$
so that
$Q_{0}^{i} = q_{0}^{i}$.
Furthermore, it is assumed that at time $0$, all scouts share the same
\emph{initial location}
$X_{0}^{i} = x_{0} \in \bbZ^{d}$;
this is typically taken to be the origin, i.e.,
$x_{0} = 0$,
however later on in the technical sections, we also consider scouts with
arbitrary initial locations.
We refer to the $5$-tuple
$\langle c, \cS, x_{0}, \mathbf{q}_{0}, \Pi \rangle$
as a \emph{scout protocol}.

The process
$(\mathbf{X}_{n}, \mathbf{Q}_{n})_{n \geq 0}$
will be referred to as a \emph{scout process} (or a \emph{$c$-scout process}).
Given a scout protocol
$\cP = \langle c, \cS, x_{0}, \mathbf{q}_{0}, \Pi \rangle$,
it is completely defined via
\[
\left( \mathbf{X}_{0}, \mathbf{Q}_{0} \right)
=
\left( x_{0}^{c}, \mathbf{q}_{0} \right)
\]
and
\[
\bbP \left(
\mathbf{X}_{n + 1} - \mathbf{X}_{n} = \Xi,
\mathbf{Q}_{n + 1} = \mathbf{q}
\mid
\left( \mathbf{X}_{m}, \mathbf{Q}_{m} \right)_{m = 0}^{n}
\right)
\, = \,
\prod_{i = 1}^{c}
\Pi \left( (Q_{n}^{i}, E_{n}^{i}), (q^{i}, \Xi^{i})\right) \, \text{a.s.}\footnote{%
Notice that both sides of the last equation are random variables.
It should be interpreted as stating that for every
$n$,
$\mathbf{x}_{0}, \dots, \mathbf{x}_{n} \in \bbZ^{d \times c}$, and
$\mathbf{q}_{0}, \dots, \mathbf{q}_{n} \in \cS^{c}$,
if the event
$\bigwedge_{m = 0}^{n}
\left( \mathbf{X}_{m}, \mathbf{Q}_{m} \right)
=
\left( \mathbf{x}_{m}, \mathbf{q}_{m} \right)$
occurs with positive probability, then for every
$\Xi \in \bbZ^{d \times c}$,
and
$\mathbf{q} \in \cS^{c}$,
the probability, conditioned on that
event, that
$\mathbf{X}_{n + 1} = \mathbf{x}_{n} + \Xi$
and
$\mathbf{Q}_{n + 1} = \mathbf{q}$
equals to
$\prod_{i = 1}^{c}
\Pi \left( q_{n}^{i}, e_{n}^{i}, q^{i}, \Xi^{i} \right)$,
where $e_{n}^{i}$ is the environment of scout $u^{i}$ with respect to the
configuration
$(\mathbf{x}_{n}, \mathbf{q}_{n})$. 
}
\]
for all
$n \geq 0$,
$\Xi \in \bbZ^{d \times c}$ and
$\mathbf{q} \in \cS^{c}$, where $\Xi^{i}$ denotes the $i$-th ($d$-dimensional)
component of $\Xi$. 

The \emph{hitting time} of grid point
$x \in \bbZ^{d}$
under $\cP$ is defined as
$\inf \{ n \geq 0 : \exists i \text{ s.t.\ } X_{n}^{i} = x \}$.
The scout protocol $\cP$ is said to be \emph{effective} if every grid point
$x \in \bbZ^{d}$
admits a finite mean hitting time.
It turns out that
$c = d + 1$
scouts are sufficient for the design of an effective
scout protocol on $\bbZ^{d}$ for every dimension
$d \geq 1$
(see \cite[Section 3.2]{EmekLSUW2015} and the related literature discussion of
the present paper).
In this paper, we prove that
$c = d + 1$
is also a necessary condition for
$d \in \{ 1, 2 \}$
and conjecture that this holds for
$d \geq 3$
as well.

\begin{theorem*}
For
$d \in \{ 1, 2 \}$,
any effective scout protocol on $\bbZ^{d}$ requires at least
$d + 1$
scouts.
\end{theorem*}

\begin{conjecture*}
For
$d \geq 3$,
any effective scout protocol on $\bbZ^{d}$ requires at least
$d + 1$
scouts.
\end{conjecture*}

\sloppy
\paragraph{Related Work.}
Graph exploration appears in many forms and is widely studied in the
CS/Mathematics literature.
In the general graph exploration setting, an \emph{agent} (or a group thereof)
is placed in some node of a given graph and the goal is to visit every node (or
every edge) by traversing the graph along its edges.\footnote{%
In some research communities, it is common to refer to the exploring agents as
\emph{robots}, \emph{ants}, or \emph{particles}.
Since the agents considered in this paper differ slightly from those
considered in the existing literature, we chose the distinctive term
\emph{scouts}.
}
There are plenty of variants and modifications of the graph exploration
problem, where one natural classification criterion is to distinguish between
\emph{directed} graph exploration~\cite{Albers2000, Deng1999}, where the edge
traversals are uni-directional, and \emph{undirected} graph
exploration~\cite{Awerbuch1999, Diks2004, Duncan2006}, where the edges can be
traversed in both directions.
The present paper falls into the latter setting.
\par\fussy

The graph exploration domain can be further divided into the case where the
nodes are labeled with unique identifiers that the agents can
recognize~\cite{Duncan2006, Panaite1998}, and the case where the nodes are
anonymous~\cite{Bender1994, Budach1978, Rollik1979}.
The exact conditions for a successful exploration also vary between different
graph exploration works, where in some papers the agents are required to halt
their exploration process~\cite{Diks2004} and sometimes also return to their 
starting position(s)~\cite{averbakh1996heuristic}.
The setting considered in the present paper requires neither halting nor
returning to the initial grid point and the nodes are anonymous (in fact,
since our scouts are controlled by finite automata that cannot ``read''
unbounded labels, unique node identifiers would have been meaningless).

A standard measurement for the efficiency of a graph exploration protocol is
its time complexity~\cite{Panaite1998}.
When the agents are assumed to operate under a synchronous schedule, this
measurement typically counts the number of (synchronous) rounds until the
exploration process is completed in the worst case.
This can be generalized to asynchronous schedules by defining a time unit as
the longest delay of any atomic action~\cite{Bender1994, chrobak2015group}.
A good example for a synchronous search problem concerning the minimization of
the worst-case search time is the widely studied \emph{cow path problem},
introduced by Baeza-Yates et al.~\cite{baezayates1993searching}.
This problem involves a near-sighted cow standing at a crossroads with ${w}$
paths leading to an unknown territory.
By traveling at unit speed, the cow wishes to discover a patch or clover that
is at distance ${d}$ from the origin in as small time as possible.
Baeza-Yates et al.\ proposed a deterministic algorithm called \emph{linear
spiral search} and showed that in the ${w=2}$ case, this algorithm will find
the goal in time at most ${9d}$.
The classic cow path problem was extended to the multiple cows setting by 
L\'{o}pez-Ortiz and Sweet~\cite{lopez2001parallel}.

Another common measure for the efficiency of a graph exploration protocol is
the size of the memory that the agents can maintain~\cite{Diks2004,
Fraigniaud2004}.
In the present paper, the execution is assumed to be synchronous and the
scouts are controlled by (probabilistic) finite automata whose memory size is
constant, independent of the distance to the target.

Graph exploration with agents controlled by finite automata has been studied,
e.g., by Fraigniaud et al.~\cite{Fraigniaud2005} who showed that a single
deterministic agent needs
$\Theta(D \log \Delta)$
bits of memory to explore a graph, where $D$ stands for the diameter and
$\Delta$ stands for the maximum degree of the input graph.
Considerable literature on graph exploration by a finite automaton agent
focuses on a special class of graphs called \emph{labyrinths}~\cite{Blum1978,
Budach1978, Doepp1971}.
A labyrinth is a graph that can be embedded on the infinite $2$-dimensional
grid and has a set of obstructed cells that cannot be entered by the agent.
The labyrinth is said to be \emph{finite} if the set of obstructed cells is
finite and the term \emph{exploration} in the context of finite labyrinths
refers to the agent getting arbitrarily far from its (arbitrary) starting
points.
It is known that any finite labyrinth can be explored by a deterministic
finite automaton agent using $4$ \emph{pebbles}, where a pebble corresponds to
a movable marker, and some finite labyrinths cannot be explored with one
pebble~\cite{Blum1977}.
While the infinite $2$-dimensional grid studied in the present paper is a
special (somewhat degenerate) case of a finite labyrinth and our scouts are
also controlled by finite automata, the goal of the scout protocols considered
here is to hit every grid point and we allow randomization.

An extensively studied special case of a single agent controlled by a
probabilistic finite automaton is that of a (simple unbiased) \emph{random
walk}.
Aleliunas et al.~\cite{Aleliunas1979} proved that for every finite
(undirected) graph, the random walk's \emph{cover time}, i.e., the time taken
to visit every node at least once, admits a polynomial mean.
Alon et al.~\cite{Alon11} studied the extension of this classic model to
multiple agents.
Among other results, they proved that in some graphs, the speed-up in the mean
cover time (as a function of the number of agents) is exponential, whereas in
other graphs, it is only logarithmic. 
Cooper et al.~\cite{cooper2009multiple} investigated the case of $c$
independent random walks exploring a random $r$-regular graph, and analyzed
the asymptotic behavior of the graph cover time and the number of steps before
any of the walks meet.

Random walks receives a lot of attention also in the context of infinite
graphs.
A classic result in this regard is that a random walk on $\bbZ^{d}$ is
\emph{recurrent} --- namely, it reaches every point with probability one --- if
and only if
$d \leq 2$;
nevertheless, even in this case, it is only \emph{null-recurrent}:
the mean hitting time of some (essentially all) points is infinite.
This gives rise to another research question asking whether the exploration
process induced by
$c > 1$
(non-interacting) random walks on $\bbZ^{d}$ is \emph{positive-recurrent},
that is, the mean hitting time of all points is finite.
It is relatively well known (can be derived, e.g., from Proposition 4.2.4 in
\cite{LawlerL2010}) that for
$d = 1$,
the minimum value of $c$ that yields a positive-recurrent exploration process
is
$c = 3$,
whereas for
$d = 2$,
the exploration process induced by $c$ random walks remains null-recurrent for
any finite
$c$.
Cast in this terminology, the research question raised in the present paper is
how many agents are required to turn the exploration process on $\bbZ^{d}$
(dealing with
$d \in \{ 1, 2 \}$)
into a positive-recurrent one, given that the
agents are augmented with a finite memory logic and local interaction skills.

The power of team exploration and the effect of the agents ability to
communicate with each other has been studied in various settings.
Fraigniaud et al.~\cite{fraigniaud2006collective} investigated the exploration
of a tree by $c$ locally interacting mobile agents and constructed a
distributed exploration algorithm whose running time is
$O (c / \log(c))$
times larger than the optimal exploration time with full knowledge of the
tree.
They also showed that for some trees, in the absence of communication, every
algorithm is $\Omega(c)$ times slower than the optimal.
Another good example for the advantages in the agents communication can be
found in the work of Bender et al.~\cite{Bender1994} showing that two
cooperating robots that can recognize when they are at the same node and can
communicate freely at all times can learn any strongly-connected directed
graph with $n$ indistinguishable nodes in expected time polynomial in $n$.

Exploration processes by agents with limited memory occur also in nature.
A prominent example for such a process is \emph{ant foraging} that was
modeled recently by Feinerman et al.~\cite{FeinermanKLS2012, FeinermanK2012} as
an abstract distributed computing task involving a team of $n$ \emph{ants}
that explore $\bbZ^{2}$ in search for an adversarially hidden treasure.
A variant of this model, where the ants are controlled by finite automata with
local interaction capabilities was studied in \cite{EmekLUW2014,
LangnerUSW2014}.

While the focus of \cite{EmekLUW2014, LangnerUSW2014}, as well as that of
\cite{FeinermanKLS2012, FeinermanK2012}, is on the speed-up as $n$ grows
asymptotically, Emek et al.~\cite{EmekLSUW2015} investigated the smallest
number of ants that can guarantee that the treasure is found in finite
time.
They studied various different settings (e.g., deterministic ants, ants
controlled by pushdown automata, asynchronous schedules), but when restricted
to ants controlled by probabilistic finite automata operating under a
synchronous schedule, their model is very similar to the one studied in
the present paper.
Cast in our terminology, they established the existence of an effective scout
protocol with $3$ scouts on $\bbZ^{2}$ and their proof can be easily extended
to design an effective scout protocol with
$d + 1$
scouts on $\bbZ^{d}$ for any fixed
$d \geq 1$
(see \cite[Theorem 3]{EmekLSUW2015}).
They also claimed that there does not exist any effective scout
protocol with $1$ scout on $\bbZ^{2}$ (see \cite[Theorem 7]{EmekLSUW2015}),
but the proof of this claim admits a significant gap;\footnote{%
Specifically, the authors of \cite{EmekLSUW2015}, who has a non-empty
intersection with the authors of the present paper, defined in Section 5.1 of
that paper a certain subset $\bbZ_{s}$ of $\bbZ$ and claimed that every point
in $\bbZ_{s}$ has a finite mean hitting time.
This claim is not substantiated in \cite{EmekLSUW2015} and we now understand
that its proof is far more demanding than the general intuition we had in mind
when \cite{EmekLSUW2015} was written.
}
the impossibility result established in the present paper regarding the
$1$-dimensional grid clearly subsumes that claim.
For a survey on generalized graph exploration problems, see, e.g.,
\cite{Fomin2008}.


\paragraph{Paper Organization.}
The primary technical contribution in this paper is the proof that there does
not exist an effective scout protocol on $\bbZ^{2}$ with $2$ scouts.
This proof is presented in
\LongVersion 
Section~\ref{section:main-proof} with some
necessary tools developed in
Section~\ref{section:hitting-time-estimates}.
\LongVersionEnd 
\ShortVersion 
Section~\ref{section:main-proof}.
\ShortVersionEnd 
It turns out that the line of arguments leading to this proof also includes, in
passing, a proof for the secondary result of this paper:
there does not exist an effective scout protocol on $\bbZ^{1}$ with $1$ scout.
For clarity, in Section~\ref{section:one-scout}, we refine the proof of this
secondary $\bbZ^{1}$ result from that of the primary $\bbZ^{2}$ case.

\section{Two Scouts on $\bbZ^{2}$}
\label{section:main-proof}
Our main goal in this section is to establish the following theorem.

\begin{theorem}
\label{thm:A}
Let
$(\mathbf{X}_n, \mathbf{Q}_n)_{n \geq 0}$
be a scout process on $\bbZ^2$ with protocol
$\cP = \langle 2, \cS, 0, \mathbf{q}_{0}, \Pi \rangle$.
Then there exists some grid point
$x \in \bbZ^2$
of which the expected hitting time is infinite, namely,
\begin{equation} \label{e:1}
\bbE \left(
\inf \{ n \geq 0 \, : \, X^1_n = x \text{ or } X^2_n = x \}
\right)
\, = \,
\infty \, .
\end{equation}
\end{theorem}

The proof of Theorem~\ref{thm:A} is presented in a top-down
fashion.
Its main ideas and arguments are introduced in
Section~\ref{section:top-level}.
The argumentation  relies on three non-trivial propositions
\LongVersion 
which are proved in  Sections \ref{section:frequent-meetings},
\ref{section:trap}, and \ref{section:no-cover}.
\LongVersionEnd 
\ShortVersion 
whose proof is deferred to the full version.
\ShortVersionEnd 
\LongVersion 
The proofs of these propositions rely on a generalization of the scout
process notion which is presented in Section~\ref{section:generalized}.
They also make use of some random walk hitting time estimates which are
relegated to section~\ref{section:hitting-time-estimates}.
\LongVersionEnd 

\subsection{Top Level Proof}
\label{section:top-level}
For the remaining of this section, let us suppose towards a contradiction
that~\eqref{e:1} does not hold.
The first step is to show that if this is the case, then the two scouts must
meet infinitely often with probability one and moreover, the distribution of
the time between successive meetings must have a stretched-exponentially
decaying upper tail.
Formally, let
$N_0 := 0$
and set 
\[
N_k
\, := \,
\inf \{ n > N_{k-1} : X_n^1 = X_n^2 \}, \quad k = 1, 2, \dots \, .
\]

\begin{proposition}
\label{prop:2}
Let
$(\mathbf{X}_n, \mathbf{Q}_n)_{n \geq 0}$
be a scout process on $\bbZ^2$ with two scouts, state space $\cS$, transition
function $\Pi$, initial position
$\mathbf{x}_0 = 0$,
and initial state
$\mathbf{q}_0 \in \cS^2$
and suppose that \eqref{e:1} does not hold.
Then, with probability one, all $N_k$ are finite.
Moreover, there exists
$\delta > 0$
such that for all
$k = 0, 1, \dots$
and
$u \geq 0$, almost surely
\begin{equation}
\label{e:5.1}
\bbP \left(
N_{k+1} - N_k > u \, \mid \, \mathbf{Q}_{N_k}
\right)
\, \leq \,
\tfrac{1}{\delta} e^{-\delta \sqrt{u}} \, .
\end{equation}
\end{proposition}

In view of Proposition~\ref{prop:2}, we now define for
$k = 0, 1, \dots$, the random variables
\begin{align*}
\mathbf{A}_{k} & := \mathbf{Q}_{N_{k}} \, , \\
Y_k & := X^1_{N_k} = X^2_{N_k} \, , \\
R_{k+1} & := N_{k+1} - N_k \, , \quad \text{and} \quad
R_{0} := 0
\end{align*}
and observe that $\mathbf{A}_{k}$ is the states of the scouts' automata at the
time of their $k$-th meeting, $Y_k$ is their position at this time, and $R_k$
is the time that elapsed from the $k-1$-st meeting.
The strong Markov property of the scout process then implies that
$(Y_k, \mathbf{A}_k, R_k)_{k=0}^\infty$
is a Markov chain on
$\bbZ^2 \times \cS^2 \times \bbZ_{\geq 1}$.
Moreover, thanks to the spatial homogeneity of the transition function of
$(\mathbf{X}_n, \mathbf{Q}_n)_{n \geq 0}$,
we further have for all
$k \geq 0$,
almost surely 
\begin{equation}
\label{e:10.2}
\bbP \left(
(Y_{k+1} - Y_k, \mathbf{A}_{k+1}, R_{k+1}) \in \cdot
\, \mid \,
(Y_m, \mathbf{A}_m, R_m)_{m=0}^k
\right)
\, = \,
\bbP \big(
(Y_{k+1} - Y_k, \mathbf{A}_{k+1}, R_{k+1}) \in \cdot
\, \mid \,
\mathbf{A}_{k}
\big) \,.
\end{equation}

Notice that $R_{k}$ is an upper bound on the maximal distance traveled by the
scouts between their $k - 1$-st and $k$-th meetings, namely, 
\begin{equation} \label{e:8}
\forall i=1,2\,,\,
n  \in [N_{k-1},N_k] :\:
\max \left\{
\left\| X^i_n - Y_{k-1} \right\| , \left\| X^i_n - Y_k \right\|
\right\}
\, \leq \,
R_{k} 
\, .
\footnote{%
Unless stated otherwise, the notation $\| \cdot \|$ is used to denote the
${\ell}_{\infty}$
norm.}
\end{equation}
In particular,
\[
\inf \{ k \geq 0 : \|Y_k - x\| \leq R_{k+1} \}
\, \leq \,
\inf \{ n \geq 0 : X^{1}_{n} = x\, \text{ or } X^{2}_{n} = x \} \, .
\]
It follows that if \eqref{e:1} is false, then for every $x \in \bbZ^2$, we
must have
\begin{equation}
\label{e:102}
\bbE \left(
\inf \{ k \geq 0 : \|Y_k - x\| \leq R_{k+1} \}
\right)
\, < \, 
\infty \, .
\end{equation}

The process
$(Y_{k}, \mathbf{A}_{k}, R_{k})_{k \geq 0}$
can be viewed as describing a single \emph{explorer} on $\bbZ^{2}$ which has
the ability to explore a ball of radius $R_{k + 1}$ around its location
$Y_{k}$ for
$k = 0, 1, \dots$
Let us first define such a process in a formal manner.

\begin{definition}
Let $\cS$ be a finite non-empty state space,
$\Pi
:
\cS \times \big(\cS \times {\bbZ}^{2} \times \bbZ_{\geq 1}\big)
\rightarrow
[0, 1]$
a (probability) transition function,
$x_0 \in \bbZ^2$,
and $q_0 \in \cS$.
An {\em explorer process} on $\bbZ^2$ with state space $\cS$,
transition function $\Pi$, initial position $x_{0}$ and initial state $q_0$ is a random process
$(X_{n}, Q_{n}, R_{n})_{n \geq 0}$
on
$\bbZ^{2} \times \cS \times \bbZ_{\geq 1}$
satisfying
\[
(X_{0}, Q_{0}, R_{0})
\, = \,
(x_{0}, q_{0}, 0)
\]
and for all
$n \geq 1$,
$\xi \in \bbZ^{2}$,
$q \in \cS$,
$r \in \bbZ_{\geq 1}$,
\[
\bbP \left(
{X}_{n + 1} - {X}_{n} = \xi,
{Q}_{n + 1} = q,
R_{n + 1}  = r
\mid
({X}_{m}, {Q}_{m}, R_{m})_{m = 0}^{n}
\right)
\, = \,
\Pi \big( Q_{n}, (q, \xi, r) \big) \, a.s.
\]
\end{definition}

By~\eqref{e:10.2}, this definition applies to the process
$(Y_{k}, \mathbf{A}_{k}, R_{k})_{k \geq 0}$.
Furthermore, by Proposition~\ref{prop:2}
and \eqref{e:8}, the conditional distributions of both
$Y_{k+1} - Y_{k}$
and $R_{k+1}$ given $\mathbf{A}_{k}$ have (at least) a stretched-exponentially
decaying upper tail (with deterministic constants).   
Proposition~\ref{prop:3} will now show that if, in addition, \eqref{e:102}
holds, then such an explorer process must eventually get trapped in a finite
set of grid points.

\begin{proposition}
\label{prop:3}
Let
$(X_{n}, Q_{n}, R_{n})_{n \geq 0}$
be an explorer process on $\bbZ^2$ such that for some
$\delta > 0$
and all
$n \geq 0$,
$u \geq 0$,
\begin{equation} \label{e:150}
\bbP \left(
\|X_{n+1} -X_n\| + R_{n+1} > u \mid Q_n
\right)
\, \leq \,
\tfrac{1}{\delta} \rme^{-\delta \sqrt{u}} \ \text{a.s.} 
\end{equation}
If for all
$x \in \bbZ^{2}$,
\begin{equation} \label{e:151}
\bbE \left( \inf \{ n \geq 0 : \|X_{n} - x\| \leq R_{n+1} \} \right)
\, < \,
\infty \, ,
\end{equation}
there must exist a stopping
time $\tau$ (for the explorer process)
and a non-random
$r < \infty$
such that 
\begin{equation}
\label{e:191}
\bbP \big(\tau < \infty \,,\,\, 
\forall n \geq \tau:\: \|X_n - X_\tau\| < r\big) = 1 \,.
\end{equation}
\end{proposition}

Now, we set
$T := N_\tau$
and observe that $T$ is a stopping
time for the process
$(\mathbf{X}_n, \mathbf{Q}_n)_{n \geq 0}$.
Therefore, Proposition~\ref{prop:2}, Proposition~\ref{prop:3} and the assumed
validity of \eqref{e:102} imply that there exist
$r <\infty$
and a stopping
time $T$, such that with probability one, after time
$T < \infty$,
the two scouts meet only inside a ball of radius $r$ around the grid point
$X^{1}_{T} = X^{2}_{T} = Y_{\tau}$.
Moreover, they keep meeting infinitely often and the times between successive
meetings have finite means.
Since this implies, in particular, that whenever a scout is away from this
ball, it evolves independently of the other scout, it follows that each scout
must always be in a grid point from which the mean time of returning to the
ball, {\em as a single scout process}, is finite.
Proposition~\ref{prop:4.5} will show that the set of such grid points is a
``small'' subset of the whole grid.

To make things precise, given
$\hat{\alpha} \in \bbR^{2}$
with
$\|\hat{\alpha}\|_{2} = 1$
and
$M > 0$,
let us define the \emph{thick ray} of width $M$ in direction $\hat{\alpha}$ as
\[
\cR(\hat{\alpha}, M)
\, := \,
\left\{
x \in \bbR^2 : \ | \langle x, \hat{\alpha}^{\perp} \rangle | < M
\, \text{ and }
\langle x, \hat{\alpha} \rangle > -M
\right\} \, ,
\]
where $\hat{\alpha}^{\perp}$ denotes (any) unit vector which is perpendicular
to $\hat{\alpha}$ and $\langle x, y \rangle$ denotes the scalar product between $x$ and $y$.
For what follows, if
$(Z_n)_{n \geq 0}$
is a Markov chain on $\cA$ and
$z_0 \in \cA$,
then we shall write
$\bbP_{z_0}(\cdot)$
for the probability measure under which $Z_0 = z_0$.
The same notation will apply to the corresponding expectation.

\begin{proposition} \label{prop:4.5}
Let $(X_{n}, Q_{n})_{n \geq 0}$ be a single scout process on $\bbZ^{2}$ 
with state space $\cS$ and transition function $\Pi$ and let also $r > 0$.
There exist $m \in \bbZ_{\geq 1}$, unit vectors
$\hat{\alpha}_{1}, \dots, \hat{\alpha}_{m} \in \bbR^{2}$
and
$M < \infty$,
such that if
$(x_0, q_0) \in \bbZ^2 \times \cS$
are such that
\begin{equation} 
\label{e:20.1}
\bbE_{(x_0, q_0)} \left(
\inf \{ n \geq 0 :\: \|X_n\| < r\}
\right)
\, < \,
\infty  
\end{equation}
then
\begin{equation} 
\label{e:21.1}
x_0 \in \cD := \bigcup_{i=1}^{m} \cR(\hat{\alpha}_{i}, M) \,.
\end{equation}
\end{proposition}

We can now finish the proof of the theorem. 
Let $(\wt{X}_n, \wt{Q}_n)_{n \geq 0}$ be a single scout process on $\bbZ^2$ with state space $\cS$ and transition function $\Pi$ as in the conditions of Theorem~\ref{thm:A}. Let also $r$ be given by Proposition~\ref{prop:3} applied to the process $(Y_k, \mathbf{A}_k, R_k)_{k \geq 0}$. Applying Proposition~\ref{prop:4.5} with $(\wt{X}_n, \wt{Q}_n)_{n \geq 0}$ and $r$, we obtain a subset $\cD \subseteq \bbR^2$ as defined in~\eqref{e:21.1}. We now claim that for $i=1,2$, 
\begin{equation} \label{e:103}
\bbP \left( X^i_{T + n} \in Y_\tau + \cD \,,\, \forall n \geq 0 \right)
\, = \,
1 \, .\footnote{%
If $x \in \bbR^d$ and $A \subseteq \bbR^d$ then $x+A$ stands for the set
$\{x+y :\: y \in A\}$.}
\end{equation}

Indeed, for
$n \geq 0$
consider the first time after $T+n$ such that scout $i$ comes to within
distance smaller than $r$ of $Y_\tau$, namely,
\[
\wh{N}_n^i := \inf \left\{ m \geq 0 :\: \| X^i_{T+n+m} - Y_\tau \| < r
\right\} \, .
\]
Since
$N_{\tau + n} \geq T+n$
and
$\|X^i_{N_{\tau + n}} - Y_\tau\| < r$,
it follows that
$T+ n + \wh{N}_n^i \leq N_{\tau + n}$.
Using the strong Markov property, the almost surely finiteness of $\tau$ and
Proposition~\ref{prop:2} (or \eqref{e:150}), we then have
\begin{align*}
\bbE \wh{N}_n^i \, \leq \, \bbE \left( N_{\tau + n} - N_{\tau} - n \right)
\, \leq \, & 
\sum_{k=1}^{n} \bbE \left(
\bbE \left( N_{\tau + k} - N_{\tau+{k-1}} \,\big|\, \mathbf{Q}_{N_{\tau+{k-1}}} \right)
\right) \\
= \, &
\sum_{k=1}^{n} \bbE \big( \bbE_{(0, \mathbf{Q}_{N_{\tau+{k-1}}})} N_1 \big)
\, < \,
\infty \, ,
\end{align*} 
where the inner expectation in the last line is with respect to the process $(\mathbf{X}_n, \mathbf{Q}_n)_{n \geq 0}$.

On the other hand,
\[
\bbE \wh{N}_n^i
\, \geq \,
\bbP \left( X^i_{T+n} \notin Y_\tau + \cD \right)
\cdot
\bbE \left( \wh{N}_n^i \,\big|\,  X^i_{T+n} \notin Y_\tau + \cD \right) \, .
\]
In light of Proposition~\ref{prop:3}, we know that
$\wh{N}^i_n \leq \inf \{ m \geq 0 :\: X^i_{T+n+m} = X^{3-i}_{T+n+m} \}$
with probability one.
From this fact and the strong Markov property, it follows that
\begin{align*}
& \bbE \left(
\wh{N}_n^i \,\big|\,  X^i_{T+n} \notin Y_\tau + \cD
\right) \\
& \, =
\bbE \left(
\left( \bbE_{(X^i_{T+n},  Q^i_{T+n})} \inf \{m \geq 0 :\: \|\wt{X}_m -
Y_\tau\| < r\} \right)
\, \big| \,
X^i_{T+n} \notin Y_\tau + \cD
\right) \, ,
\end{align*}
where the inner expectation is just with respect to the process $(\wt{X}_n, \wt{Q}_n)_{n \geq 0}$.
But by Proposition~\ref{prop:4.5}, under the conditioning the inner
expectation is infinite almost surely.
We therefore must have that
$\bbP (X^i_{T+n} \notin Y_\tau + \cD) = 0$.
Summing over all $n$ and using the union bound yields \eqref{e:103}.

Finally, the validity of \eqref{e:103} for
$i = 1, 2$
implies that with probability one,
\[
X^i_n \in \rmB_T \cup \big( Y_\tau + \cD \big)
\quad \forall i = 1, 2 \, \text{ and } \, n = 0, 1, \dots \, ,
\]
where henceforth $\rmB_\rho$ is the closed ball of radius $\rho$ around $0$ in the $\ell^\infty$ norm.
Since
$T < \infty$,
the random set
$\rmB_T \cup \big(Y_{\tau} + \cD\big)$
is always a strict subset of $\bbZ^2$.
It follows that there must exist
$x \in \bbZ^2$
such that with positive probability, $x$ will not be reached by either scout.
This in turn shows that \eqref{e:1} holds, in contradiction to what we have
assumed in the first place, thus establishing Theorem~\ref{thm:A}.

\subsection{Single and Generalized Scout Processes and Underlying Automaton}
\label{section:generalized}
If $(X_n, Q_n)_{n \geq 0}$ is a single scout process, its environment vector $E_n$ is always $0^\cS$ and hence one can effectively consider a reduced version $\Pi' : \cS \times (\cS \times \{-1,0,+1\}^d) \mapsto [0,1]$ of the transition function $\Pi$ of the underlying protocol, where
\[
\Pi' \big(q,(q', \xi) \big) = \Pi \big( (q,0^\cS), (q',\xi)\big) 
\ , \ \ q, q' \in \cS \,,\,\, \xi \in \{-1,0,+1\}^d \,.
\]
The marginal process $(Q_n)_{n \geq 0}$ then becomes a Markov chain on $\cS$ by itself (with transition function given by $(q,q') \mapsto \sum_{\xi} \Pi'\big(q, (q', \xi)\big)$) and we shall refer to this process as the scout's (underlying) {\em automaton}. In particular, the {\em irreducible classes} of the automaton will be the irreducible states classes of the Markov chain $(Q_n)_{n \geq 0}$ and the automaton will be called {\em irreducible} in this Markov chain is such. 

In the sequel, we shall also need to consider a (single) {\em generalized scout process} $(\wh{X}_n, \wh{Q}_n)_{n \geq 0}$, which is defined as the single scout process above, but with more general steps. Formally, for $d \geq 1$, the process, which now takes values in $\bbR^d \times \cS$, is defined exactly as before, only that the transition function now takes the form $\wh{\Pi} : \cS \times (\cS \times \cW) \mapsto [0,1]$, where $\cW$ is some discrete nonempty subset of $\bbR^d$. As in the case of a single scout process, the process $(\wh{Q}_n)_{n \geq 0}$, which is Markovian by itself, will be referred to as the underlying automaton. The following lemma will be used more than once in what follows.

\begin{lemma}
\label{lem:12}
Let $(\wh{X}_n, \wh{Q}_n)_{n \geq 0}$ be a generalized single scout process on $\bbR^d$ for $d \geq 1$ with state space $\cS$, transition function $\wh{\Pi}$, initial position $x_0 =0$ and some initial state $q_0 \in \cS$. Suppose also that its automaton is irreducible and set
\[
	T := \inf \{n > 0 :\: \wh{Q}_n = q_0\} 
	\ , \ \ 
	\zeta := \wh{X}_T \,.  
\]
If $\bbP (\zeta = 0) = 1$, then there exist $r \in (0, \infty)$ and a finite subset 
$\cA \subseteq \rmB_r \subset \bbR^d$ such that
\begin{equation}
\label{e:11.1}
\bbP \big( \forall n \geq 0 :\: \wh{X}_n \in \cA \big) = 1  \,.
\end{equation}
and for all $x \in \cA$ and $u \in \bbZ_{\geq 0}$,
\begin{equation}
\label{e:11.2}
\bbP \big( \forall 0 \leq n \leq u :\: \wh{X}_n \neq x \big) \leq C \rme^{-C' u} \,.
\end{equation}
with some $C, C' > 0$.
\end{lemma}
\begin{proof}
We will show that for each $q \in \cS$ there exists $x_q \in \bbR^d$ such that almost surely for all $n$,
\begin{equation}
\label{e:14}
\wh{Q}_n = q \ \Rightarrow \wh{X}_n = x_q \,.
\end{equation}
This will prove~\eqref{e:11.1} with $\cA := \{x_q :\: q \in \cS\}$ and $r = \max \{\|x_q\| :\: q \in \cS\}$.
To this end, observe first that if $\zeta = 0$, then with probability one 
\begin{equation}
\label{e:170}
\wh{Q}_n = q_0 \ \Rightarrow \ \wh{X}_n =0 \,.
\end{equation}
Suppose now that there are $x_q \neq x_q' \in \bbZ^d$ and $n_1, n_2 \geq 0$ such that
\begin{equation}
\label{e:171}
\bbP \big(\wh{Q}_{n_1} = q \,,\,\, \wh{X}_{n_1} = x_q \big) > 0 \ , \ \    
\bbP \big(\wh{Q}_{n_2} = q \,,\,\, \wh{X}_{n_2} = x'_q \big) > 0 \,.
\end{equation} 
Since $(\wh{Q}_n)_{n \geq 0}$ is irreducible and hence recurrent, it follows from~\eqref{e:170},\eqref{e:171} 
and the Markov property that there exist $m_1, m_2 \geq 0$ such that
\begin{equation}
\label{e:172}
\bbP_{(x_q, q)} \big(\wh{Q}_{m_1} = q_0 \,,\,\, \wh{X}_{m_1} = 0 \big) > 0 \ , \ \ 
\bbP_{(x_q', q)} \big(\wh{Q}_{m_2} = q_0 \,,\,\, \wh{X}_{m_2} = 0 \big) > 0 \,.
\end{equation} 
But then using the product rule with the first event in~\eqref{e:171} and the second event in~\eqref{e:172} and the Markov property again, we get
\[
\bbP \big(\wh{Q}_{n_1 + m_2} = q_0 \,,\,\, \wh{X}_{n_1 + m_2} = x_q - x'_q \big) > 0 \,.
\]
Since $x_q - x'_q \neq 0$, this leads to a contradiction to~\eqref{e:170}
occurring with probability $1$. 

To show~\eqref{e:11.2}, we use the irreducibility and finiteness of the Markov chain $(\wh{Q}_n)_{n \geq 0}$ together with standard theory to claim that the hitting time of any $q \in \cS$ has an exponentially decaying upper tail. Then~\eqref{e:11.2} follows by virtue of~\eqref{e:14}.

\end{proof}

\LongVersion 
\subsection{Two Scouts on $\bbZ^2$ Must Meet Frequently}
\label{section:frequent-meetings}
In this section we prove Proposition~\ref{prop:2}. We shall do this gradually,
first assuming that the two scouts evolve independently as single scout processes and have irreducible automata (in the sense discussed in beginning of Subsection~\ref{section:generalized}). We shall then remove the irreducibility restriction and finally
consider the full (dependent) two-scout process.

Let therefore $(X^1_n, Q^1_n)_{n \geq 0}$ and $(X^2_n, Q^2_n)_{n \geq 0}$ be two independent single scout processes. For $i=1,2$, suppose that scout process $i$ has state space $\cS^i$, transition function $\Pi^i$, initial position $x^i_0 \in \bbZ^2$ and initial state $q^i_0 \in \cS^i$. As before, we shall write $\mathbf{X}_n$ for $(X^1_n, X^2_n)$, $\mathbf{Q}_n$ for $(Q^1_n, Q^2_n)$, etc. We also define, 
\[
	N := \inf\{n > 0: X_n^1 = X_n^2\} 
\]
and for $y \in \bbZ^2$,
\[
\tau_y := \inf \{ n\geq 0 :\: X^1_n = y \text{ or } X^2_n = y\} \,.
\]

The first lemma deals with the case of irreducible automata.
\begin{lemma}
\label{lem:19}
Let $(X^1_n, Q^1_n)$ and $(X^2_n, Q^2_n)$ be two independent single scout processes on $\bbZ^2$ with state spaces $\bm{\cS}$ and transition functions $\mathbf{\Pi}$, starting from initial positions $\mathbf{x}_0$ and initial states $\mathbf{q}_0$. Suppose also that the automaton of each scout processes is irreducible.
Then either 
\begin{equation}
\label{e:10}
	\big| \big\{y \in \bbZ^2 :\: \bbE (N \wedge \tau_y) = \infty \big\} \big| = \infty
\end{equation}
or for all $u \geq 0$,
\begin{equation}
\label{e:11}
	\bbP (N > u) \leq \tfrac{1}{\delta} \rme^{-\delta (\sqrt{u} - \|x_0^1 - x_0^2\|)} \,.
\end{equation}
for some $\delta > 0$ which depends only on the transition
functions $\mathbf{\Pi}$.
\end{lemma}
\begin{proof}
We shall observe each scout $i$ at times of successive returns to $q^i_0$ as well as both scouts together at times of successive simultaneous returns to states $\mathbf{q}_0=(q^1_0, q^2_0)$. This will turn the scouts processes into random walks.

Formally, let $T_0^i = 0$, $T_0^\Delta = 0$ and for $k=1, \dots$ define:
\[
 \begin{split}
 T_k^i &:= \inf \{n \geq T_{k-1}^i : Q_n^i = q^i_0\} \quad i=1,2 \,, \\
 T_k^\Delta & := \inf \{n \geq T_{k-1}^\Delta : Q_n^1 = q^1_0, Q_n^2 = q^2_0\} \,.
 \end{split}
\]
For $n \geq 0$, and with $\rme_1 = (1,0) \in \bbZ^2$, we also set:
\[
		S_{n}^i := \langle X_{T_n^i}^i, \rme_1 \rangle \quad i=1,2 
		\quad, \qquad
		S_{n}^\Delta :=  \langle X_{T_n^\Delta}^1 -  X_{T_n^\Delta}^2 \, , \rme_1 \rangle \,.
\]
Then for $i=1,2,\Delta$ and $k \geq 1$ we let:
\[
s^i_0 := S^i_0 
\ , \ \ 
\zeta^i_k:= S^i_k - S^i_{k-1}
\ , \   \ 
\nu_k^i := T_k^i - T_{k-1}^i 
\ , \ \ 
R_k^i = 2\nu_k^i \,.
\]

For $i=1,2$, the Markov property and spatial homogeneity of the process $(X^i_n, Q^i_n)$  imply that the triplets $(\zeta^i_k, \nu_k^i, R_k^i)_{k \geq 1}$ are i.i.d. and consequently that $(S^i_n, T^i_n)_{n \geq 0}$ is a random walk. Since the process $(\mathbf{X}_n, \mathbf{Q}_n)_{n \geq 0}$ can be viewed as a single scout process on $\bbZ^4$, the same applies to $(\zeta^\Delta_k, \nu^\Delta_k, R_k^\Delta)_{k \geq 1}$ and $(S^\Delta_n, T^\Delta_n)_{n \geq 0}$. Moreover, since the underlying scout automata are irreducible, standard Markov chain theory implies that 
there exists $C > 0$ such that for $i=1,2,\Delta$,
\begin{equation}
\label{e:35.1}
\bbP (\nu^i_1 > u) \leq C^{-1} \rme^{-C u} \,.
\end{equation}
Then, since $|\zeta^i_1| \leq R^i_1 = 2\nu^i_1$, we also have
\begin{equation}
\label{e:35.2}
\bbP (|\zeta^i_1| + \nu^i_1 + R^i_1  > u) \leq \delta^i \rme^{-\delta^i u} \,.
\end{equation}
for some $\delta^i > 0$. 

For $i=1,2,\Delta$, let the {\em effective} drift of walk $i$ be defined as
\[
d^i := (\bbE \zeta^i_1)/(\bbE \nu^i_1)  \,.
\]
We next claim that 
\[
d^\Delta = d^1 - d^2 \,.
\]
Setting $X^\Delta_n := X^1_n - X^2_n$, it is enough to  show that for $i=1, 2, \Delta$, with probability one,
\begin{equation}
\label{e:31}
	\lim_{m \to \infty} \langle X^i_m, \rme_1 \rangle /m = d^i \,.
\end{equation}
To this end, for all $m \geq 0$, we set
\[
K^i_m := \sup \{n \geq 0: \: T^i_n \leq m\} \,.
\]
By the law of large numbers, with probability one, for all $\epsilon > 0$, there exists $n_0$ such that for all $n > n_0$, 
\begin{equation}
\label{e:33}
n(\bbE \zeta^i_1 - \epsilon)  \leq S^i_n \leq n(\bbE \zeta^i_1 + \epsilon) 
\end{equation}
and
\begin{equation}
\label{e:34}
n(\bbE \nu^i_1 - \epsilon) \leq T^i_n \leq n(\bbE \nu^i_1 + \epsilon) \,.
\end{equation}

Since $m \in [T^i_{K^i_m}, T^i_{K^i_m+1})$ by definition, for all $m$ large enough~\eqref{e:34} implies
\begin{equation}
\label{e:35}
m/(\bbE \nu^i_1 + 2\epsilon) \leq K^i_m \leq m/(\bbE \nu^i_1 - \epsilon) \,.
\end{equation}
But then from~\eqref{e:33},
\[
m \frac{\bbE \zeta^i_1 - \epsilon}{\bbE \nu^i_1 +2\epsilon} \leq S^i_{K^i_m} \leq m \frac{\bbE \zeta^i_1 + \epsilon}{\bbE \nu^i_1 - \epsilon} \,.
\]
Since this is true for all $\epsilon > 0$, this shows that as $m \to \infty$, almost-surely 
\begin{equation}
\label{e:36}
S^i_{K^i_m}/m \to (\bbE \zeta^i_1)/(\bbE \nu^i_1) = d^i \,.
\end{equation}
 At the same time,
\[
	\big| \langle X^i_m, \rme_1 \rangle - S^i_{K^i_m} \big| \leq T^i_{K^i_m+1} - T^i_{K^i_m} \,.  
\]
Dividing by $m$ both sides above, taking $m \to \infty$ and using~\eqref{e:34} and~\eqref{e:35} we obtain
\[
\lim_{m \to \infty} \big| \langle X^i_m ,\rme_1 \rangle /m - S^i_{K^i_m}/m \big| = 0 \,.
\]
Together with~\eqref{e:36} this shows~\eqref{e:31}.

Care must be taken to handle the possible degeneracy of the steps of the walks $(S^i_n, T^i_n)_{n \geq 0}$, for $i=1,2,\Delta$. We therefore first assume that 
\begin{equation}
\label{e:25}
\bbP (\zeta^\Delta_1 = 0) < 1 \,,
\end{equation}
and relegate the treatment of the complementary case to a later point in the proof. As for $i=1,2$, if $\zeta^i_1/\nu^i_1$ is a non-random constant, then we must have $\zeta^i_1 = d^i \nu^i_1$ with probability $1$. Applying then Lemma~\ref{lem:12} for the (single, irreducible one dimensional) generalized scout process $(\wh{X}^i_n, \wh{Q}^i_n)_{n \geq 0}$, defined by setting 
\begin{equation}
\wh{X}^i_n := \big \langle X^i_n - x^i_0 ,\, \rme_1 \big \rangle - d^i n
\quad, \qquad
\wh{Q}^i_n := Q^i_n \,,
\quad; \qquad n=0, \dots, 
\end{equation}
we obtain $r^i \in (0,\infty)$ such that almost surely, 
$|\wh{X}^i_n| \leq r^i$ for all $n \geq 0$. It follows that 
\begin{equation}
\bbP \big( \forall n\geq 0 : \big|\langle X^i_n - x^i_0,\, \rme_1 \rangle - d^i n \big| \leq r^i \big) = 1 \,,
\end{equation}
and accordingly we redefine $(\zeta^i_k, \nu^i_k, R^i_k)_{k\geq 1}$ and $(S^i_n, T^i_n)_{n \geq 0}$ as
\begin{equation}
\forall k \geq 1:\: \zeta^i_k := d^i \,,\,\, \nu^i_k = 1
\,,\,\, R^i_k = r^i
 \quad; \qquad
\forall n \geq 0:\: S^i_n := s_0^i  + d^i n \,,\,\, T^i_n = n  \,.
\end{equation}
Observe that under the new definition, $(S^i_n, T^i_n)_{n \geq 0}$ is still a random walk, albeit with deterministic steps which trivially satisfy~\eqref{e:35.2} for some $\delta^i > 0$. Moreover $K^i_m = m$.

With the above definitions, the process $(S^\Delta_n, R^\Delta_n)_{n \geq 0}$ is a look-around random walk of the type considered in sub-section~\ref{sub:7.1} and the processes $(S^i_n, T^i_n, R^i_n)_{n \geq 0}$ for $i=1,2$, form a pair of time-varying look-around random-walks of the type considered in subsection~\ref{sub:7.2}. Moreover, by the definition of $R^i_n$, we have
\begin{equation}
\label{e:26}
\forall i=1,2, m \geq 0 :\:
\big| \langle X^i_m, \rme_1 \rangle - S^i_{K^i_m} \big| \leq R^i_{K^i_m+1} \,.
\end{equation}

We now appeal to Proposition~\ref{prop:22} with the random walks $(S^1_n, T^1_n, R^1_n)_{n \geq 0}$ and $(S^2_n, T^2_n, R^2_n)_{n \geq 0}$ and to Lemma~\ref{lem:17} with the random walk $(S^\Delta_n, R^\Delta_n)_{n \geq 0}$. Let $\rho > 0$ be as given by the first proposition, $\tau_\rho$ be as in~\eqref{e:4.3} for the walk $S^\Delta_n$ and set $N^\Delta_\rho := T^\Delta_{\tau_\rho}$.
Suppose first that $\bbP (N^\Delta_\rho < \infty \,,\,\, N^\Delta_\rho < N) >0$. Then there exist $n \geq 0$ and $s^i_1 \in \bbZ$ for $i=1,2,\Delta$ such that~\eqref{e:2.2.1} holds and 
\begin{equation}
\label{e:130}
\bbP \big( N > N^\Delta_\rho = n \,,\,\, S_n^i = s_1^i \,: i=1,2, \Delta \big) > 0 \,.
\end{equation}
For such $s_1^i$, by Proposition~\ref{prop:22}, there exist $-\infty < x < y < \infty$ such that $|y-x|> 2$ and ~\eqref{e:21} holds. But then for each $z \in \bbZ^2$ satisfying $\langle z, \rme_1 \rangle \in [x,y]$ and $\|z - x^1_0\| > n$ and $\|z - x^2_0\|  > n$, by the Markov property, the fact that $N$ and $N_\rho^\Delta$ are stopping times and~\eqref{e:26}, 
\[
\begin{split}
\bbE \big( N \wedge \tau_z \, \big|\, N > N^\Delta_\rho = n \,,\,\, & S_n^i = s_1^i \,: i=1,2, \Delta \big) \\
& \geq n + \bbE_{(s_1^1, s_1^2)} \big( \min\{\sigma, \tau_{[x,y]}^1, \tau_{[x,y]}^2 \} 
\big) = \infty \,,
\end{split}
\]
where $\tau^i_{[x,y]}$ is defined as in~\eqref{e:100}. In light of~\eqref{e:130} shows~\eqref{e:10}.

Otherwise, $\bbP  (N \leq N^\Delta_\rho) = 1$ and then by Lemma~\ref{lem:17} (which is in force due to~\eqref{e:25}) and~\eqref{e:35.1}, for all $u \geq 0$ 
\[
\begin{split}
 \bbP(N > u) & \leq \bbP(N^\Delta_\rho > u ) \\
	 &  \leq \bbP_{s_0^\Delta} (\tau^\Delta_\rho > \sqrt{u} ) 
	 +  \sum_{m=1}^{\lfloor \sqrt{u} \rfloor} \bbP \big( \nu^\Delta_m > \sqrt{u} \big) \\
 	& \leq \delta^{-1} \rme^{-\delta (\sqrt{u} -|s_0^\Delta|)} + C \sqrt{u} \rme^{-C' \sqrt{u}}
 	\leq \frac{1}{\delta'} \rme^{-\delta' (\sqrt{u}-\|x_0^1 - x_0^2\|)} \,,
\end{split}
\]
for some $C, C' > 0$. This shows~\eqref{e:11}.

It remains to treat the case when~\eqref{e:25} does not hold. In this case, we first replace $\rme_1=(1,0)$ by $\rme_2 := (0,1)$ in the entire argument. If~\eqref{e:25} now holds, then we proceed as before and the proof is complete. If not, then we must have $(X^1_{T^\Delta_1} - X^2_{T^\Delta_1}) - (x^1_0 - x^2_0) = 0$ with probability $1$. But then, by Lemma~\ref{lem:12} applied to the generalized scout process
\begin{equation}
(\wh{X}_n, \wh{Q}_n)_{n \geq 0} := \big((X^1_n - X^2_n)-(x^1_0 - x^2_0), (Q^1_n, Q^2_n) \big)_{n \geq 0} \,,
\end{equation}
there exists $\cA \subseteq \bbZ^2$ such that~\eqref{e:11.1} and~\eqref{e:11.2} hold. 

If $-(x^1_0 - x^2_0) \in \cA$, then by~\eqref{e:11.2} we have $\bbP (N > u) \leq C \rme^{-C' u}$ for all $u \geq 0$, which in particular shows~\eqref{e:11}. If not, then $N = \infty$ with probability one and hence almost surely 
\begin{equation}
\label{e:33.1}
\forall z \in \bbZ^2 :\: N \wedge \tau_z = \tau_z \,.
\end{equation}
Thanks to Lemma~\ref{lem:6} or Lemma~\ref{lem:7} (depending on whether $d^1 \neq 0$ or $d^1 = 0$) applied to the look-around random walk $(S^1_n, R^1_n)_{n \geq 0}$, as defined above, there exist $x^1 \in \bbR$, $\alpha^1 \in \{-1, +1\}$ and $C > 0$ such that for all $z \in \bbZ^2$ with $\alpha^1(\langle z, \rme_1 \rangle - x^1) > 0$ and $u \geq 1$ 
\begin{equation}
\bbP \big( \forall n \leq u :\: X^1_n \neq z \big)
\geq  \bbP \big( \forall n \leq u :\: \big|S^1_n - \langle z,\,\rme_1 \rangle \big| >  R^1_{n+1} \big)
\geq C/\sqrt{u} \,.
\end{equation} 

By a similar argument, there exists there exist $x^2 \in \bbR$, $\alpha^2 \in \{-1, +1\}$ and $C' > 0$ such that for all $z \in \bbZ^2$ with $\alpha^2(\langle z, \rme_2 \rangle - x^2) > 0$ and $u \geq 1$, 
\begin{equation}
\bbP \big( \forall n \leq u :\: X^2_n \neq z \big)
\geq C'/\sqrt{u} \,.
\end{equation} 
It follows then by the independence of the scout processes that for all $z \in \bbZ^2$ satisfying $\alpha^1(\langle z, \rme_1 \rangle - x^1) > 0$ and $\alpha^2(\langle z, \rme_2 \rangle - x^2) > 0$ and $u \geq 1$,
\begin{equation}
\bbP \big( \tau_z > u )\geq C C'/u \,.
\end{equation} 
The tail formula for expectation then gives $\bbE \tau_z = \infty$ for all such $z$. Since there infinitely many such $z$-s and thanks to~\eqref{e:33.1}, we obtain~\eqref{e:10}.
\end{proof}

Next, we remove the restriction to irreducible automata.
\begin{lemma}
\label{lem:110}
Let $(X^1_n, Q^1_n)$ and $(X^2_n, Q^2_n)$ be two independent scout processes on $\bbZ^2$ with state spaces $\bm{\cS}$ and transition functions $\mathbf{\Pi}$, starting from $x^1_0 = x^2_0 = 0$ and initial states $\mathbf{q}_0$. Then either
\begin{equation}
	\label{e:105}
		\big| \big\{y \in \bbZ^2 :\: \bbE(N \wedge \tau_y) = \infty \big\} \big| = \infty
	\end{equation}
or there exists $\delta > 0$ such that for all $u \geq 0$.
\begin{equation}
\label{e:106}
	\bbP (N > u) \leq \tfrac{1}{\delta} \rme^{-\delta \sqrt{u}} \,.
\end{equation}
\end{lemma}
\begin{proof}
For $i=1,2$, let $\cS^i_1, \dots, \cS^i_{m^i} \subset \cS^i$ for some $m^i \geq 1$ be the
recurrent irreducible classes of automaton $(Q^i_n)_{n \geq 0}$. Note that if $q_0^i \in \cS^i_l$ for some $l \in \{1, \dots, m^i\}$ then
$(X^i_n, Q^i_n)$ is also a single scout process with state space $\cS^i_l$ and
transition function $\Pi^i_l$, which is the proper restriction of $\Pi^i$ to
$\cS^i_l$.
Moreover, its automaton then is irreducible. Define,
\[
	 \sigma^i := \inf \{n \geq 0: Q_n^i \in \cup_{l=1}^{m^i} \cS^i_l\} \quad, i=1,2 \,,
\]
and set $\sigma := \sigma^1 \vee \sigma^2$. By standard Markov chain theory,
\[
	\bbP(\sigma^i > u) \leq C^{-1} \rme^{-C u}
		\ , \ \ i = 1,2,
\]
which implies the same for $\sigma$. In particular $\sigma, \sigma^1$ and $\sigma^2$ are finite almost surely.

Suppose first that there exist $n \geq 0$, $l^1$, $l^2$, $\mathbf{x} \in \bbZ^4$ and $\mathbf{q} \in \cS^1_{l^1} \times \cS^2_{l^2}$ such that 
\begin{equation}
\label{e:38}
	\bbP \big( N > \sigma = n \,,\,\, \mathbf{X}_n =\mathbf{x} \,,\,\, \mathbf{Q}_n = \mathbf{q} \big) > 0 
\end{equation}
and that~\eqref{e:10} holds in Lemma~\ref{lem:19} applied to $(\mathbf{X}_n, \mathbf{Q}_n)_{n \geq 0}$ as two independent scout processes with transition functions $\Pi^1_{l^1}$, $\Pi^2_{l^2}$ respectively and with $\mathbf{x}_0 = \mathbf{x}$ and $\mathbf{q}_0 = \mathbf{q}$.
Since $N$ and $\sigma$ are stopping times, it then follows that for any $y \in \bbZ^2$ with $\|y\| > n$,
\begin{equation}
\label{e:39}
\bbE ( N \wedge \tau_y \,|\, N > \sigma = n, \mathbf{X}_n = \mathbf{x}, \mathbf{Q}_n = \mathbf{q} \big)
	= n + \bbE_{(\mathbf{x}, \mathbf{q})} ( N \wedge \tau_y ) = \infty \,,
\end{equation}
Together with~\eqref{e:39} this gives~\eqref{e:105}.

Otherwise, by Lemma~\ref{lem:19}, anytime~\eqref{e:38} holds, we must also have~\eqref{e:11} with some $\delta^{(l^1, l^2)} > 0$, which depends only the transition functions $\Pi^{l^1}$ and $\Pi^{l^2}$. Setting
\[
\delta_0 = \min\{\delta^{(l^1, l^2)} :\: l^1 =1, \dots, m^1 \, , \, 
	l^2 = 1, \dots, m^2 \} \,
\]
we then have with some $C >0$,
\[
\begin{split}
\bbP(N > u) & \leq 
\bbP \big(N > u, \sigma \leq \sqrt{u}/4 \big) +  \bbP \big(\sigma > \sqrt{u}/4 \big)\\
& \leq \bbE \Big(\bbP_{(\mathbf{X}_\sigma, \mathbf{Q}_\sigma)} (N > u) 1_{\{\sigma \leq \sqrt{u}/4\}} \Big)  +  C^{-1} \rme^{-C \sqrt{u}} \\
& \leq \delta_0^{-1} \rme^{-\delta_0 (\sqrt{u} - \sqrt{u}/2)} + C^{-1} \rme^{-C \sqrt{u}}
   \leq \delta^{-1} \rme^{-\delta \sqrt{u}} \,,
\end{split}
\]
where we have used the strong Markov property and the fact that $\|X^1_\sigma - X^2_\sigma\| \leq 2\sigma$. This shows~\eqref{e:106}.
\end{proof}

\begin{proof}[Proof of Proposition~\ref{prop:2}]
The proof follows by induction on $k$. For $k=1$, since up to time $N$, the two-scout process evolves as two independent single scout processes, if $y$ and $\mathbf{q}_0$ are such that $\bbE_{(\mathbf{0}, \mathbf{q}_0)} \big( N \wedge \tau_y \big) = \infty$ for two independent scouts, then also 
\begin{equation}
\label{e:121}
\bbE_{(\mathbf{0}, \mathbf{q}_0)} \big( \tau_y \big) \geq \bbE_{(\mathbf{0}, \mathbf{q}_0)} \big( N \wedge \tau_y \big) = \infty \,,
\end{equation}
for the two-scout process. Lemma~\ref{lem:110} with $(\cS^1, \Pi^1) = (\cS^2, \Pi^2) = (\cS, \Pi)$,
and the assumption that~\eqref{e:1} does not hold, imply then that $\mathbf{q}_0$ must be such that~\eqref{e:106} holds. Since $N_1 = N$, this gives the case $k=1$. 

Suppose now that~\eqref{e:5.1} holds up to $k-1$. In particular, this shows
that $N_{k-1} < \infty$ almost surely.
Conditioning on $\cF_{N_{k-1}}$,\footnote{%
Throughout, for a stopping time $Z$, we use the standard notation $\cF_{Z}$ to
denote the sigma-algebra generated by $Z$.
}
the strong Markov property and the spatial
homogeneity of the underlying processes imply that almost surely,
\[
\bbP \big( N_k - N_{k-1} \in \cdot \,|\, \cF_{N_{k-1}} \big) =
\bbP_{(\mathbf{0}, \mathbf{Q}_{N_{k-1}})} \big( N \in \cdot \big) \,.
\]
If with positive probability $\mathbf{Q}_{N_{k-1}}=\mathbf{q}_0$ for $\mathbf{q}_0$ such that~\eqref{e:105} holds, then since the number of vertices visited by both scouts up to time $N_{k-1}$ is finite, it follows as in~\eqref{e:121}, that~\eqref{e:1} cannot hold. We therefore must have for all $u \geq 0$,
\[
\bbP \big( N_k - N_{k-1} > u \,|\, \cF_{N_{k-1}} \big) \leq \delta^{-1} \rme^{-\delta \sqrt{u}} 
\  \text{a.s.}
\]
The tower property for conditional expectation shows~\eqref{e:5.1} \,.
\end{proof}
\LongVersionEnd 

\LongVersion 
\subsection{One Explorer Must Eventually Get Trapped}
\label{section:trap}
In this section we prove Proposition~\ref{prop:3}. As in the case of a single scout process, if 
$(X_n, Q_n, R_n)_{n \geq 0}$ is an explorer process with state space $\cS$, then $(Q_n)_{n \geq 0}$ is a Markov chain on $\cS$, to be referred to as the explorer's {\em automaton}. We begin by assuming that this automaton is irreducible.
\begin{lemma}
\label{lem:70}
Let $(X_n, Q_n, R_n)_{n \geq 0}$ be an explorer process on $\bbZ^2$ with state space $\cS$ and transition function $\Pi$,  starting from $x_0 = 0$ and some $q_0 \in \cS$. Suppose also that its automaton is irreducible and that there exists $\delta > 0$ such that for all $q \in \cS$, $n \geq 0$ and $u \geq 0$,
\begin{equation}
\label{e:16a}
	\bbP \big( \|X_{n+1}-X_n\| + R_{n+1} > u \, \big| \, Q_n \big) \leq \tfrac{1}{\delta} 
		\rme^{-\delta \sqrt{u}} \,. 
\end{equation}
If it holds that
\begin{equation}
\label{e:160}
\big| \big\{ x \in \bbZ^2 :\: \bbE \big(\inf \{ n \geq 0 :\: \|X_n - x\| \leq R_{n+1} \}\big) = \infty \big\} \big| < \infty \,,
\end{equation}
then there must exist $r > 0$, which depends only on $\Pi$, such that 
\[
\bbP \big( \forall n \geq 0 :\: \|X_n\| < r \big) = 1 \,.
\]
\end{lemma}
\begin{proof}
Let $T_0 = 0$ and for $k=1, \dots$, set:
\begin{equation}
\label{e:67}
\begin{split}
T_k &:= \inf \{n > T_{k-1} :\: Q_n = q_0 \} \, \\
\zeta_k &:= X_{T_k} -  X_{T_{k-1}} \, \\
\nu_k &:= T_k-T_{k-1} \, \\
R'_k & := \sum_{n={T_{k-1}}}^{T_k-1} R_{n+1} \,.
\end{split}
\end{equation}
Since the Markov chain $(Q_n)_{n \geq 0}$ is irreducible and hence recurrent, all times $T_k$ are finite almost surely. The Markov property then implies that the triplets $(\zeta_k, \nu_k, R'_k)_{k \geq 1}$ are i.i.d. Moreover, from standard Markov chain theory, there exists $C > 0$ such that for all $u \geq 0$
\[
\bbP ( \nu_k > u ) \leq C^{-1} \rme^{-C u} \,.
\]
In addition~\eqref{e:16a} and the strong Markov property implies that for all $m \geq 1$ and $u \geq 0$, 
\[
\bbP \big( \|X_{T_{k-1} + m} - X_{T_{k-1} + m-1}\| +  R_{T_{k-1} + m} > u \big) \leq \delta^{-1} \rme^{-\delta \sqrt{u}} \,.
\]
Then for all $k \geq 1$ and $u \geq 0$, by the union bound,
\[
\begin{split}
\bbP (\|\zeta_k\| + R'_k > u) & \leq \bbP \big( \nu_k \geq \sqrt{u} \big) +
	\sum_{m=1}^{\lfloor \sqrt{u} \rfloor} \bbP \big( \|X_{T_{k-1} + m} - X_{T_{k-1} + m-1}\| + R_{T_{k-1} + m} > \sqrt{u} \big) \\
	& \leq C \sqrt{u} \rme^{-C' u^{1/4}}
	\leq \tfrac{1}{\delta'} \rme^{-\delta' u^{1/4}}
\end{split}
\]
for some $\delta' > 0$.

Setting $S_n = \langle X_{T_n} ,\, \rme_1 \rangle$ for $n \geq 0$, we observe that the process $(S_n, R'_n)$ is of the type handled by Lemma~\ref{lem:100}. Moreover, if~\eqref{e:160} holds, then it is also true that
\[
	\big| \big\{ x \in \bbZ :\: \bbE \big( \inf \{ n \geq 0 :\: |S_n - x| < R'_{n+1} \} \big) = \infty \big\}\big| < \infty \,,
\]
Then by Lemma~\ref{lem:100}, we must have $\langle \zeta_1, \rme_1 \rangle = 0$ with probability one. Repeating the same with $S_n = \langle X_{T_n} ,\, \rme_2 \rangle$, we get $\langle \zeta_1, \rme_2 \rangle = 0$. Thus, we conclude that $\zeta_1 = 0$ with probability one. Invoking then Lemma~\ref{lem:12} for the process $(X_n, Q_n)_{n \geq 0}$, which is, in particular, a generalized scout process, the proof is complete.
\end{proof}

The proof of Proposition~\ref{prop:3} is now straightforward,
\begin{proof}
Let $\tau$ be the first time the Markov chain $(Q_n)_{n \geq 0}$ enters a recurrent class in $\cS$. By standard Markov chain theory $\tau$ is is finite almost surely. Conditional on $\cF_\tau$, by the Markov property, $(X_{\tau+n}, Q_{\tau+n}, R_{\tau+n})_{n \geq 0}$ is an explorer process with an irreducible automaton. If~\eqref{e:151} holds then there could be at most $\tau$ random vertices $x \in \bbZ^2$ such that,
\[
\bbE \big(\inf \{ n \geq 0 :\: \|X_{\tau+n} - x\| \leq R_{\tau+n+1}  \} \,|\, \cF_\tau \big)= \infty \quad \text{a.s} \,,
\]
Thus by Lemma~\ref{lem:70}, there exists a finite $R \in \cF_{\tau}$ such that,
\[
	\bbP \big( \forall n \geq 0:\: \|X_{\tau+n} - X_{\tau} \| < R \,|\, \cF_\tau \big) = 1 \quad \text{a.s.}
\]
Since the number of recurrent classes is finite, we may replace $R$ above by a non-random $r > 0$. Taking expectation, we recover~\eqref{e:191}. 
\end{proof}
\LongVersionEnd 

\LongVersion 
\subsection{One Trapped Scout Cannot Cover the Whole Space}
\label{section:no-cover}
In this section we prove Proposition~\ref{prop:4.5}. One more time, we start with the case of an irreducible automaton.
\begin{lemma}
\label{lem:13}
Let $\big((X_n, Q_n) :\: n \geq 1\big)$ be a single scout process on $\bbZ^2$ with state space $\cS$ and suppose that its automaton is irreducible. Let also $r < \infty$. There exist $\hat{\alpha} \in \bbR^2$ with $\|\hat{\alpha}\|_2 = 1$ and $M < \infty$ such that if $x_0 \in \bbR^2$ and $q_0 \in \cS$ are such that
\begin{equation}
\label{e:80}
\bbE_{(x_0, q_0)} \big( \inf \big\{ n \geq 0 :\: \|X_n\| < r \big\}\big) < \infty 
\end{equation}
then
\[
x_0 \in \cR(\hat{\alpha}, M) \,.
\]
\end{lemma}
\begin{proof}
Let $x_0 \in \bbZ^2$ and $q_0 \in \cS$ and suppose that the scout process starts from position $x_0$ and state $q_0$. Defining $T_k$, $\zeta_k$ and $\nu_k$ as in~\eqref{e:67}, we have that $(\zeta_k, \nu_k)_{k \geq 1}$ are i.i.d. and that
\begin{equation}
\label{e:68a}
\bbP ( \nu_k > u ) \leq C^{-1} \rme^{-C u} \,,
\end{equation}
for some $C > 0$.
In particular, 
\[
S_n := x_0 + \sum_{k=1}^n \zeta_k = X_{T_n} \, , \,\, n=0,\dots \,,
\]
is a random walk on $\bbZ^2$ starting from $x_0$. Now define,
\[
\hat{\alpha} := \left\{ 
	\begin{array}{lr}
		\bbE (\zeta_1)/\|\bbE (\zeta_1)\|_2	&\text{if } \bbE (\zeta_1) \neq 0 \,, \\
		0									&\text{if } \bbE (\zeta_1) = 0 \,,
	\end{array} \right.
\]
and let $\alpha^\perp$ be any unit vector which is perpendicular to $\hat{\alpha}$. Setting also for $n,k \geq 0$,
\[
\begin{split}
	\zeta^\alpha_k & := \langle \zeta_k, {\hat{\alpha}}\rangle \ \ \ , \
	S^\alpha_n := \langle S_n, {\hat{\alpha}}\rangle \ \ \ , \ 
	x^\alpha_0 := \langle x_0, \hat{\alpha} \rangle \,, \\	
	\zeta^\perp_k & := \langle \zeta_k, {\alpha}^\perp \rangle \ , \
	S_n^\perp := \langle S_n, {\alpha}^\perp \rangle \ ,\ 
	x^\perp_0 := \langle x_0, \alpha^\perp \rangle \,, \\
\end{split}
\]
we see that both $(S_n^\alpha)_{n \geq 0}$ and $(S_n^\perp)_{n \geq 0}$ are random walks on $\bbR$ with steps $\zeta^\alpha_k$ and $\zeta^\perp_k$ and initial positions $x^\alpha_0$ and $x^\perp_0$, respectively. Moreover, by definition $\bbE \zeta_1^\perp = 0$.

If $\bbP(\zeta_1^\alpha = 0) = 1$, then Lemma~\ref{lem:12} applied to the generalized scout process
$(\langle X_n, \hat{\alpha} \rangle, Q_n)_{n \geq 0}$ implies the existence of $r^\alpha < \infty$ such that
$|\langle X_n, \hat{\alpha} \rangle - x_0^\alpha| < r^\alpha$ for all $n \geq 0$ with probability $1$. If we therefore set
\[
R_n^\alpha := \left\{ 
	\begin{array}{ll}
		r + \nu_n		& \text{if } \bbP (\zeta_1^\alpha = 0) < 1 \,, \\
		r + r^\alpha	& \text{if } \bbP (\zeta_1^\alpha = 0) = 1 \,,
	\end{array}
	\right. 
\]
then~\eqref{e:80} implies that
\begin{equation}
\label{e:69}
\bbE \big(\inf \{n \geq 0 :\: \|S^\alpha_n\| < R^\alpha_{n+1} \}\big) < \infty \,.
\end{equation}
A similar argument shows that for some $r^\perp < \infty$ and with
\[
R_n^\perp := \left\{ 
	\begin{array}{ll}
		r + \nu_n		& \text{if } \bbP (\zeta_1^\perp = 0) < 1 \,, \\
		r + r^\perp		& \text{if } \bbP (\zeta_1^\perp = 0) = 1 \,,
	\end{array}
	\right. 
\]
if~\eqref{e:80} holds, then also
\begin{equation}
\label{e:69a}
\bbE \big( \inf \{n \geq 0 :\: \|S^\perp_n\| < R^\perp_{n+1} \} \big) < \infty \,.
\end{equation}

Using Lemma~\ref{lem:6} if $\bbE \zeta^\alpha_1 > 0$ or Lemma~\ref{lem:18} if $\bbE \zeta^\alpha_1 = 0$, for the process $(S^\alpha_n, R^\alpha_n)_{n \geq 0}$, noting that~\eqref{e:68a} implies that condition~\eqref{e:70} is in force, we see that for~\eqref{e:69} to hold we must have
\begin{equation}
\label{e:89}
x_0^\alpha < M^\alpha \,,
\end{equation}
for some $M^\alpha > 0$.
Similarly, since $\bbE \zeta^\perp_1 = 0$ and~\eqref{e:68a} holds, we may use Lemma~\ref{lem:18} for the process $(S^\perp_n, R^\perp_n)_{n \geq 0}$ to obtain $M^\perp > 0$ such that if~\eqref{e:69a} holds then
\begin{equation}
\label{e:90}
|x_0^\perp| < M^\perp \,.
\end{equation}
Combining~\eqref{e:89} and~\eqref{e:90} we have $x_0 \in \cR(-\hat{\alpha}, M)$ for 
$M := \max\{M^\alpha, M^\perp\}$ as desired. 
\end{proof}

\begin{proof}
For $m \geq 1$, let $\cS_1,\cS_2,...,\cS_m$ be the recurrent irreducible state
classes of the automaton $(Q_n)_{n \geq 0}$.
If $q_0 \in \cS_l$ for some $l \in \{1, \dots, n\}$, then the process $(X_n,
Q_n)_{n \geq 0}$ is a single scout process with state space $\cS_l$ and
transition function $\Pi^l$ which is the proper restriction of $\Pi$ to
$\cS^l$.
Moreover, such a scout process has an irreducible automaton.
Therefore, by Lemma~\ref{lem:13} there exists $\hat{\alpha}_l$ and $M_l$ such
that if~\eqref{e:20.1} holds then
\begin{equation}
\label{e:200}
x_0 \in \cR(\hat{\alpha}_l, M_l) \,.
\end{equation}

For any other $q_0 \in \cS$, there exists $l \in \{1, \dots, m\}$, $q' \in \cS^l$, $n \leq |\cS|$ and $x' \in \bbZ^2$ with $\|x'-x_0\| \leq n$
 such that 
\[
\bbP_{(x_0, q_0)} \big(X_n = x' \,,\,\, Q_n = q') > 0
\]
If~\eqref{e:20.1} holds, then by the Markov property and~\eqref{e:200} we must have that $x' \in \cR(\hat{\alpha}_l, M_l)$ which implies that $x_0 \in \cR(\hat{\alpha}_l, M_l + |\cS|)$. 
Taking $M := \max_{l \leq m} M_l + |\cS|$ we obtain~\eqref{e:21.1} as desired.
\end{proof}
\LongVersionEnd 

\LongVersion 
\section{Random Walk Hitting Time Estimates}
\label{section:hitting-time-estimates}
In this section we state and prove the random walk estimates, which are used in the proof of the main theorems.

\subsection{One Random Walk with a Stretched Exponential Look-Around}
\label{sub:7.1}
Fix $\delta > 0$ and 
let $\big((\zeta_k, R_k) :\: k=1,\dots \big)$ be a sequences of discrete i.i.d. random pairs, taking values in $\bbR \times [1,\infty)$ and satisfying 
\begin{equation}
\label{e:70}
	\bbP(|\zeta_1| + R_1 > u) \leq \tfrac{1}{\delta} \rme^{-u^\delta} 
	\,,\,\, u \geq 0 \,.
\end{equation}
Notice that we do not insist that for a given $k$, the random variables $\zeta_k, R_k$ are independent of each other. We shall think of $\zeta_k$ as the spatial displacement of a random walk in the $k$-th step and of $R_k$ as the radius of a ball (interval) around the position of the walk at time $k-1$, inside which the walk is allowed to ``peek'' at this time. We therefore fix also $s_0 \in \bbR$ and define for all $n \geq 0$, 
\[
	S_n := s_0 + \sum_{k=1}^n \zeta_k \ , \ \ 
\]
The process $(S_n, R_n)_{n \geq 0}$ will be referred to as a {\em look-around} random walk. 

\begin{lemma}
\label{lem:6}
Let $(S_n, R_n)_{n \geq 0}$ be the look-around random walk process defined
above and suppose that~\eqref{e:70} holds.
If $\bbE(\zeta_1) > 0$, then there exists $r > 0$, such that for all
$x < s_0 - r$,
\begin{equation}
\label{e:201}
\bbP \big(\inf\{n \geq 0 : |S_n - x| \leq R_{n+1}\} = \infty \big) > 0 \, .
\end{equation}
\end{lemma}
\begin{proof}
Let $\mu := \bbE(\zeta_k) > 0$. By the Strong Law of Large numbers $\frac{S_n}{n}  \longrightarrow \mu$  as $n \to \infty$ almost surely. Therefore we may find $m$ large enough such that
\begin{equation}
\label{e:17}
	\bbP \big(S_n > \mu n/2 \,:\,\, n \geq m ) > 1/2 \,.
\end{equation}
On the other hand, by union bound, for possibly larger $m$ we have,
\begin{equation}
\label{e:19}
\bbP \big( \exists n \geq m \,:\,\, R_{n+1} > n^{1/3} \big)  \leq
		\sum_{n=m}^{\infty}\frac{1}{\delta} \rme^{-n^{\delta/3}}  \leq 1/4 \,.
\end{equation}
Finally, since $\min\{ S_n - R_{n+1} \ : n = 1, \dots, m\} > -\infty$ with probability $1$, it follows that we can find $x_0$ small enough such that for all $x \leq x_0$, 
\begin{equation}
\label{e:20}
	\bbP \big( |S_n - x| \leq R_{n+1} \,:\,\, n = 0, \dots, m \big) \leq 1/8 \,.
\end{equation}
The event in~\eqref{e:201} is implied by the intersection of the event on the left hand side of~\eqref{e:17} with the complements of the events on the left hand side of~\eqref{e:19} and~\eqref{e:20}. The above shows that this has positive probability as required.
\end{proof} 

Next we deal with the case $\bbE(\zeta_1) = 0$.
\begin{lemma}
\label{lem:7.5}
Let $(S_n, R_n)_{n \geq 0}$ be the look-around random walk process defined
above and suppose that~\eqref{e:70} holds.
If
$\bbE(\zeta_1) = 0$,
but
$\bbP(\zeta_1 = 0) < 1$,
then there exists $C > 0$, such that for all $u \geq 1$ and $x \in \bbR$ satisfying $0 < x - s_0  < u^{1/4}$,
\begin{equation}
\label{e:78.5}
\bbP\big(\inf\{n \geq 0 : S_n + R_{n+1} \geq x\} \geq u \big) 
\leq C\frac{x - s_0} {\sqrt{u}}.
\end{equation}
\end{lemma}
\begin{proof}
The probability on the left hand side of~\eqref{e:78.5} is bounded above by
\[
\bbP\big(\inf\{n \geq 0 : S_n \geq x \} \geq u \big) 
\]
and the result follows immediately from Theorem 5.1.7 of~\cite{LawlerL2010}.
\end{proof}

The opposite direction is given by
\begin{lemma}
\label{lem:7}
Let $(S_n, R_n)_{n \geq 0}$ be the look-around random walk process defined above and suppose that~\eqref{e:70} holds. Suppose also that $\bbE(\zeta_1) = 0$ and if $\bbP(\zeta_1 = 0) = 1$, then 
$\bbP(R_1 < M) = 1$ for some $M > 0$. Then there exist $r > 0$ and $C > 0$, such that for all $u \geq 1$ and $x \geq s_0+r$,
\begin{equation}
\label{e:78}
\bbP\big(\inf\{n \geq 0 : S_n + R_{n+1} \geq x\} \geq u \big) 
\geq \frac{C}{\sqrt{u}}.
\end{equation}
\end{lemma}
\begin{proof}
For simplicity of the argument, we shall assume that $\zeta_1$ is supported on the integers. The argument for a general (discrete, Real) distribution is similar.
Without loss of generality we can assume that $s_0 = -1$ and $x = r-1$, where $r > 0$ will be determined later. If $\bbP(\zeta_1 = 0) = 1$, then~\eqref{e:78} will clearly hold, once $r = M+1$. Otherwise, set
\[
\eta_x := \inf \{n \geq 0 :\: S_n \geq x \} \,.
\]

From Theorem 5.1.7 of~\cite{LawlerL2010} we have that for all $n \geq 0$ and $x \in [0, n^{1/4}]$,
\begin{equation}
\label{e:91}
C (x+1) /\sqrt{n} \leq \bbP\big(\eta_x > n \big) \leq C' (x+1) /\sqrt{n} 
\end{equation}
At the same time from Theorem~7 of~\cite{alio1999wiener}, we have that for all $y \in [-n^{1/4}, -1]$,
\[
\bbP\big(\eta_0 > n  \,,\,\, S_n = y \big) 
	\leq C |y| / n^{3/2} \,.
\]
Then for all $k_0 \geq 1$ and $n > k_0$,
\begin{equation}
\label{e:93}
\begin{split}
\bbP \big( & \eta_0 > n \,,\,\, \exists k \in [k_0, n] \:: S_k \geq -(\log k)^{3/\delta} \big)  \\
& \leq \sum_{k=k_0}^{n} \sum_{y=-(\log k)^{3/\delta}}^{-1}
	 \bbP \big (\eta_0 > k \,,\,\, S_k = y \big) \bbP \big( \eta_{-y} > n-k \big) \\
& \leq C \Big( \sum_{k=k_0}^{n/2} k^{-3/2} (\log k)^{9/\delta} n^{-1/2}
 + \sum_{k=n/2}^{n-n^{1/2}} n^{-3/2} (\log n)^{9/\delta} (n-k)^{-1/2}  + n^{1/2} n^{-3/2} (\log n)^{9/\delta} \Big) \\ 
& \leq C n^{-1/2} k_0^{-1/3} \,. 
\end{split}
\end{equation}

Now let $K := \inf \big \{k \geq 1 :\: |\zeta_m| + R_{m+1} \leq (\log (m \vee k))^{2/\delta}
	\,,\forall m \in [0,n] \big \}$.
Then for all $k_1 \geq 2$ and $n > k_1$,
\begin{equation}
\label{e:96}
\begin{split}
\bbP \big( & \eta_0 > n \,,\,\, K \in [k_1, n/2) \big) = \sum_{k=k_1}^{n/2} \big( \eta_0 > n \,,\,\, K = k \big)  \\
& \leq \sum_{k=k_1}^{n/2} \bbP \Big( (\log (k-1))^{2/\delta} < \min_{m \leq k-1} \big( |\zeta_m| + R_{m+1} \big) \leq \max_{m \leq k} \big( |\zeta_m| + R_{m+1} \big) \leq (\log k)^{2/\delta} \,,  \eta_0 > n \Big) \\ 
& \leq C \sum_{k=k_1}^{n/2} k \rme^{-C'(\log (k-1))^2}
	\, \frac{k(\log k)^{2/\delta} +1}{\sqrt{n-k}}  \leq C n^{-1/2} \rme^{-C''(\log (k_1))^2}\,.
\end{split}
\end{equation}
Above to derive the one before last inequality, we have conditioned on $\cF_k$ and used the Markov property together with~\eqref{e:70},~\eqref{e:91}, union bound and the fact that
\[
\big \{\max_{m \leq k} \big( |\zeta_m| + R_{m+1} \big) \leq (\log k)^{2/\delta} \big \} \subseteq
\big\{ S_k \geq -k (\log k)^{2/\delta} \big\} \,.
\]
Also by union bound and~\eqref{e:70},
\begin{equation}
\label{e:97}
\begin{split}
\bbP (K > n/2) & \leq \bbP \big( |\zeta_m| + R_{m+1} > (\log (n/2))^2 \,:\,\, m \in [0, n] \big) \\
& \leq C n \rme^{-C' (\log n)^2} \leq \rme^{-C'' (\log n)^2} \,,
\end{split}
\end{equation}
Combining~\eqref{e:96} and~\eqref{e:97} we get 
\begin{equation}
\label{e:64}
\bbP \big( \eta_0 > n \,,\,\, K \geq k_1 \big) \leq n^{-1/2} \rme^{-C'(\log (k_1))^2}
	+ \rme^{-C'' (\log n)^2} \,.
\end{equation}

Subtracting the probabilities on the left hand sides of~\eqref{e:93},~\eqref{e:64} 
with $k_0 = k_1$ from the probability in~\eqref{e:91} for $x=0$, and using the derived upper bounds for the former and the lower bound for the latter, we have
\[
\begin{split}
\bbP \big(S_m & < -(\log m)^{3/\delta} \,:\, m \in [k_0, n] \ , \ \ 
	|\zeta_m| + R_{m+1} < (\log (m \vee k_0))^{2/\delta} 
	\,:\, m \in [0, n] \big) \\
& \geq  n^{-1/2} \big( C - C' k_0^{-1/3} - \rme^{-C''(\log k_0)^2} - n^{1/2} \rme^{-C''' (\log n)^2} \big)
\end{split}
\]
We may now find $k_0$ large enough such that for all $n$ large enough, the right hand side above will be at least $C n^{-1/2}$ for some $C>0$. But the event on the left hand side implies that
\[
\big\{ S_m + R_{m+1} \leq x \,:\, m \in [0,n] \big\}
\]
for all $x > k_0 (\log k_0)^{2/\delta}$. Setting $r=k_0 (\log k_0)^{2/\delta}$, this shows~\eqref{e:78} for all $x \geq r$ as desired.
\end{proof}

The following is an immediate consequence of Lemma~\ref{lem:7}.

\begin{lemma}
\label{lem:18}
Let
$(S_n, R_n)_{n \geq 0}$
be the look-around random walk process defined above and suppose that~\eqref{e:70} holds, 
$\bbE(\zeta_1) = 0$ and if $\bbP(\zeta_1 = 0) = 1$,
then $\bbP(R_1 < M) = 1$ for some $M > 0$. 
Then there exist $r > 0$, such that if $|s_0 - x| > r$, then
\begin{equation}
\label{e:78.1}
\bbE \big( \inf\{n \geq 0 : |S_n - x| \leq R_{n+1}\} \big) = \infty \,.
\end{equation}
\end{lemma}
\begin{proof}
By reversing the steps, it is enough to show the result for $x \geq s_0 + r$. But then, using the tail formula for expectation, we sum~\eqref{e:78} from $u=1$ to $\infty$ and conclude~\eqref{e:78.1} for all such $x$ if $r$ is large enough.
\end{proof}

\begin{lemma}
\label{lem:100}
Let $(S_n, R_n)_{n \geq 0}$ be the look-around random walk process defined above. Suppose that~\eqref{e:70} holds. Then if 
\begin{equation}
\label{e:72}
	\big|\big\{ x \in \bbR :\: \bbE \big( \inf \{ n \geq 0 :\: |S_n - x| \leq R_{n+1} \} \big) = \infty
		\big\}\big| < \infty \,,
\end{equation}
we must have
\[
\bbP (\zeta_1 = 0) = 1 \,.
\]
\end{lemma}
\begin{proof}
If $\bbE \zeta_1 \neq 0$, then by Lemma~\ref{lem:6} we have that~\eqref{e:72} must be false. At the same time, if $\bbE \zeta_1 = 0$ and $\bbP (\zeta_1 = 0) < 1$, then by Lemma~\ref{lem:18} the inequality in~\eqref{e:72} must be again false. It follows that $\bbP(\zeta_1 = 0) = 1$ as required.
\end{proof}

In the next two lemmas, the ``look-around'' feature of the random walk is not used.
The first one includes standard hitting time estimates for random walks. We omit the proof, as it is standard. For $\rho > 0$, let
\begin{equation}
\label{e:4.3}
 \tau_{\rho} :=
  \begin{cases}
    \inf \{n \geq 0: |S_n| > \rho\},	& \text{if} \quad \bbE \zeta_1 = 0 \,,\\
    \inf \{n \geq 0: S_n > \rho\} ,& \text{if} \quad \bbE \zeta_1 > 0  \,, \\
    \inf \{n \geq 0: S_n < -\rho\} ,& \text{if} \quad \bbE \zeta_1 < 0  \,.
  \end{cases}
\end{equation} 
Then,
\begin{lemma}
\label{lem:17}
Let $(S_n, R_n)_{n \geq 0}$ be the look-around random walk defined above and suppose that $\bbP(\zeta_1 = 0) < 1$. Then for all $\rho > 0$, there exists $\delta > 0$, such that for all $s_0 \in \bbR$, 
\[
\bbP \big( \tau_\rho > u) \leq \delta^{-1} \rme^{-\delta (u - |s_0|)} \,.
\]
\end{lemma}

The next lemma is the result of a standard application of the exponential Chebychev inequality.
\begin{lemma}
\label{lem:50}
Let $(S_n, R_n)_{n \geq 0}$ be the look-around random walk described above and suppose that $\bbE \zeta_1 = s_0 = 0$. Then for all $\mu > 0$, there exists $\delta > 0$, such that if $n \geq 0$ and $y \geq \mu n$ then,
\[
\bbP ( S_n \geq y ) \leq \rme^{-\delta y} \,.
\]
\end{lemma}
\begin{proof}
Let $L(t) := \log \bbE \rme^{t \zeta_1}$ be the log moment generating of $\zeta_1$, whose existence 
and (Real) analyticity in a neighborhood of $0$ is guaranteed by condition~\eqref{e:70}. Since $L(0) = 0$ and $L'(0) = \bbE \zeta_1 = 0$. It follow by Taylor expansion of $L$ around $0$, that for any $\mu > 0$, we may find $t_0 > 0$ and $\delta_0 > 0$, such that
\[
L(t_0) - \mu t_0 < -\delta_0 \,.
\]
Then for $n$ and $y$ as in the conditions of the lemma, by the exponential Chebychev inequality, 
\[
\begin{split}
\bbP \big( S_n \geq y \big) 
& \leq \exp \{ n L(t_0) - n \mu t_0 - (y - n \mu) t_0 \} \\
& \leq \exp \{ -\delta_0 n - (y - n \mu) t_0 \}
\leq \exp \{ -\delta y \}  \,,
\end{split} 
\]
where $\delta = \min \{t_0, \delta_0 / \mu \}$.
\end{proof}

\subsection{Two Time-Varying Random Walks} 
\label{sub:7.2}
In this sub-section we consider two random walks of the type considered in sub-section~\ref{sub:7.1}, only that in addition, each walk also has a ``time'' component. The latter can be used to ``synchronize'' between the two walks.

As in the previous subsection, we fix $\delta > 0$ and for $i=1,2$, let $\big((\zeta^i_k, \nu^i_k, R^i_k) :\: k=1,\dots \big)$ be two independent sequences of i.i.d. discrete random triplets, taking values in $\bbR \times \bbZ_{\geq 1} \times [1,\infty)$ and satisfying the following conditions:
\begin{enumerate}
\item $\bbP(|\zeta_1^i| + \nu_1^i + R_1^i > u) \leq \tfrac{1}{\delta} \rme^{-u^\delta}$.
\item If $\zeta_1^i / \nu_1^i$ is non-random, then $\bbP(\nu^i_1 = 1) = 1$ and $\bbP(|R^i_1| < 1/\delta) = 1$.
\end{enumerate}

Again, we do not insist that for a given $i$ and $k$, the random variables $\zeta^i_k, \nu^i_k, R^i_k$ are independent of each other. We shall think of $\nu^i_k$ as the time it took walk $i$ to make step $k$ and of $R^i_k$ as its ``look-around'' radius (in the sense of the previous subsection). For $i=1,2$, we now fix also $s^i_0 \in \bbR$ and define the random walk $\big((S^i_n, T^i_n) :\: n=0,1,\dots\big)$ by
\[
	S^i_n := s^i_0 + \sum_{k=1}^n \zeta^i_k \ , \ \ 
	T^i_n := \sum_{k=1}^n \nu^i_k ;
	\qquad n = 0,1,\dots \,.
\]
Setting also $R^i_0=0$, the resulting process $\big((S^i_n, T^i_n, R^i_n) :\: n=0,1,\dots \big)$ will be referred to as a {\em time-varying look-around} random walk. Its {\em effective drift} is then given by
\begin{equation}
\label{e:99}
	d^i := (\bbE \zeta^i_1) / (\bbE \nu^i_1). 
\end{equation}

We now define various hitting times. First, to translate back ``time''to number of steps, we set for $i=1,2$ and $m \geq 0$ the random variable,
\[
K^i_m := \sup \, \{k \geq 0 :\: T^i_k \leq m \} \,.
\]
Now, for a subset $A \subseteq \bbR$, the hitting time of $A$ by $S^i_n$ is 
\begin{equation}
\label{e:100}
\tau_A^i := \inf \big\{ m \geq 0 :\: \rmd(S^i_{K^i_m}, A) \leq R^i_{k^i_m+1} \big\} \,,
\end{equation}
where for $x \in \bbR$ we use the usual definition $\rmd(x, A) := \inf \{\|x-y\| :\: y \in A\}$. 
The first time the look-around balls (intervals) of the walks intersect, is defined via
\[
	\sigma := \inf \big\{ m \geq 0: \: \big|S^1_{K_m^1} - S^2_{K_m^2}\big| \leq 
		R^1_{K_m^1+1} + R^2_{K_m^2+1} \big \} \,. 
\]
Finally, we also set $s^\Delta_0 := s^1_0 - s^2_0$ and $d^\Delta : = d^1 - d^2$.

The following proposition is the main product of this subsection.
\begin{proposition}
\label{prop:22}
For $i=1,2$, let the time-varying look-around random walks $\big((S^i_n, T^i_n, R^i_n) :\: n=0,1\dots \big)$ be as defined above. Then, we may find $\rho > 0$, such that whenever:
\begin{equation}
\label{e:2.2.1}
\begin{array}{lll}
|s^\Delta_0| > \rho \quad & \text{ if } & d^\Delta = 0 \,, \\
s^\Delta_0 > \rho \quad & \text{ if } & d^\Delta > 0 \,, \\
s^\Delta_0 <- \rho \quad & \text{ if } & d^\Delta < 0 \,,
\end{array}
\end{equation}
there exist $x,y \in \bbZ$ satisfying $-\infty < x < y - 2 < \infty$, such that
\begin{equation}
\label{e:21}
\bbE \min \big(\sigma, \tau_{[x,y]}^1, \tau_{[x,y]}^2 \big) = \infty \,.           
\end{equation}
\end{proposition}
\begin{proof}
To simplify the arguments below, we shall assume that $\zeta^i_1$, $R^i_1$ are supported on $\bbZ$. The argument for general (discrete, Real) distributions is analogous. The proof follows by case-analysis of $d^1$ and $d^2$. 
By switching between walk $1$ and $2$ or negating the two walks simultaneously, we do not lose any generality by considering only the cases:
\begin{itemize}
\item $d^1 \geq 0$, $d^2 \leq 0$,
\item $d^1 > d^2 > 0$,
\item $d^1 = d^2 > 0$.
\end{itemize}

If $d^1 \geq 0$ and $d^2 \leq 0$, then $d^\Delta \geq 0$ and $\bbE \zeta^1_1 \geq 0$, $\bbE \zeta^2_1 \leq 0$. We define now for $A \subset \bbR$, $i=1,2,\Delta$ the hitting time of $A$ by walk $i$ as
\[
N_A^i := \inf \big\{ n \geq 0 :\: \rmd(S^i_n, A) \leq R^i_{n+1} \big \} \, .
\]
This is analog to $\tau^i_A$ only that time is measured in steps.

Then, by Lemma~\ref{lem:6} and Lemma~\ref{lem:7}, we may find $\rho > 0$ large enough such that if $s^\Delta_0 = s^1_0 - s^2_0 > \rho$, then there exist $x, y \in \bbZ$ such that $s_0^1 < x < y-2 < s_0^2$ and for all $u \geq 0$, 
\[
\bbP\big( N_{[x,y]}^1 > u \big) \geq \frac{C}{\sqrt{u}} \ \  , \ \ \ 
\bbP\big( N_{[x,y]}^2 > u \big) \geq \frac{C'}{\sqrt{u}} \,,
\]
Since $\tau_{[x,y]}^i \geq N_{[x,y]}^i$, it follows that for all $u \geq 0$, also
\[
\bbP\big( \tau_{[x,y]}^1 > u \big) \geq \frac{C}{\sqrt{u}} \ \  , \ \ \ 
\bbP\big( \tau_{[x,y]}^2 > u \big) \geq \frac{C'}{\sqrt{u}} \,.
\]
Now, it is not difficult to see that for such $s_0^\Delta$, 
\[
\sigma \geq \min \{\tau_{[x,y]}^1 \,,\,\, \tau_{[x,y]}^2 \} \,.
\]
It then follows that
\[
\bbE \min \{\sigma, \tau_{[x,y]}^1, \tau_{[x,y]}^2 \} = \bbE \min \{\tau_{[x,y]}^1, \tau_{[x,y]}^2 \} \,.
\]
Since the two walks are independent, by the tail formula for expectation,
\[
\begin{split}
\bbE \min \{\sigma, \tau_{[x,y]}^1, \tau_{[x,y]}^2 \}
 = \sum_{u=1}^{\infty} \bbP(\tau_{[x,y]}^1 \geq u) \bbP(\tau_{[x,y]}^2 \geq u) 
 \geq \sum_{u=1}^{\infty}\frac{C C'}{u} = \infty  \,.
\end{split}
\]

Moving to the case $d^1 > d^2 > 0$. Here $d^\Delta > 0$. By the law of large numbers, for all $\epsilon > 0$, we may find $n_0$ large enough, such that with probability at least $1 - \epsilon$ for all $n \geq n_0$, $i=1,2$, the following holds:
\[
\begin{split}
	n (\bbE \zeta^i - \epsilon) & \leq S^i_n \leq n (\bbE \zeta^i + \epsilon) \,, \\
	n (\bbE \nu^i - \epsilon) & \leq T^i_n \leq n (\bbE \nu^i + \epsilon) \,.
\end{split}
\]
At the same time, by increasing $n_0$ if needed, we have
\[
	\bbP \big( \exists n \geq n_0 :\: R^i_n > n^{1/3} \big)
	\leq \sum_{n=n_0}^\infty \delta^{-1} \rme^{-\delta n^{1/3}}
	\leq \rme^{-C n_0^{1/3}} \leq \epsilon \,.
\]
Combining the last three displays, for all $\epsilon > 0$, we may find $m_0$, such that with probability at least $1-\epsilon$, for all $m \geq m_0$, 
\begin{equation}
\label{e:110}
	m (d^i - \epsilon) \leq S^i_{K^i_m} \leq m (d^i + \epsilon)
	\ , \ \ 
	R^i_{K^i_m} \leq \epsilon m \,.
\end{equation}
By finiteness almost surely of the random variables $K^i_{m_0}$, $|\zeta^i_n|$ + $R^i_n$, for all $i=1,2$ and $n \geq 1$, for any $m_0$ and $\epsilon > 0$, we may find $r > 0$ large enough, such that with probability at least $1-\epsilon$,
\begin{equation}
\label{e:300}
\sum_{n=1}^{K^i_{m_0}+1} |\zeta^i_n| + R^i_n < r \,.
\end{equation}

Combining the above, we first choose $0 < \epsilon < \min\{d_1 - d_2,\, d_2,\, 1\}/5$, then find $m_0$ large enough and finally $r > 0$ such that both~\eqref{e:110} and~\eqref{e:300} hold with probability at least $1-\epsilon$. It follows that if $s_0^\Delta > 3r$ and $x,y \in \bbZ$ satisfy $x +2 < y  < s^2_0 - 2r$, then with probability at least $1-2\epsilon$
\[
	\tau^1_{[x,y]} = \infty \,,\,\, \tau^2_{[x,y]} = \infty \,,\,\, \sigma = \infty \,,
\]
which, of course, implies~\eqref{e:21} \,.

Turning to the hardest case $d_1 = d_2 =: d> 0$. For $\rho > 2$ to be determined later, let us assume that $s^\Delta_0 > \rho$ and without loss of generality also that $s^2_0 = 0$. For $i=1,2$ and $n \geq 0$, set
\[
Y^i_n := S^i_n - d T^i_n = s^i_0 + \sum_{k=1}^n \big(\zeta^i_k - d \nu^i_k\big) \,.
\]
Let also $z := s_0^1 / 2$ and define
\[
\sigma^1 := \inf \{n \geq 0 :\: Y^1_n - d \nu^1_{n+1} - R^1_{n+1} \leq z \} \ , \ \ 
\sigma^2 := \inf \{n \geq 0 :\: Y^2_n + R^2_{n+1} \geq z \} \,.
\]
Observe that 
\[
\begin{split}
\big\{\sigma^1 > n \big\} \, & \subseteq \, \big\{ S^1_{K^1_m} - R^1_{K^1_m+1} > d m + z 
\,:\, m=0, \dots, T^1_n \big\}  \,,\\
\big\{\sigma^2 > n \big\} \, & \subseteq \, \big\{ S^2_{K^2_m} + R^2_{K^2_m+1} < d m + z  \,:\, m=0, \dots, T^2_n \big\} \,.
\end{split}
\]
Consequently, 
\[
\big\{\sigma^1 > n \,,\, \sigma^2 > n \big\} \subseteq  \big\{\sigma > T^1_n \wedge T^2_n \big\}
	\subseteq \big\{\sigma > n \big\} \,.
\]
Similarly if $x,y \in \bbZ$ such that $x + 2 < y < 0$, then
\[
\big\{\sigma^1 > n \,,\, \sigma^2 > n \,,\, N^2_{[x,y]} > n \big\}
	\subseteq \big\{\sigma > n \,,\, \tau^1_{[x,y]} > n \,,\,  \tau^2_{[x,y]} > n \big\} \\
	= \big\{\min \{\sigma, \tau^1_{[x,y]}, \tau^2_{[x,y]}\} > n \big\} \,.
\]
Thanks to the independence between the two walks and the tail formula for expectation, to show~\eqref{e:21}, it is therefore enough to prove
\begin{equation}
\label{e:141}
\sum_{n=0}^\infty \bbP \big( \sigma^1 > n) \bbP (\sigma^2 > n \,, N^2_{[x,y]} > n \big) = \infty \,.
\end{equation}

To this end, first observe that~\eqref{e:99} implies that $\bbE \zeta^i_1 - d \nu^i_1 = 0$. This makes $Y^i_n$ a random walk with $0$ drift starting from $s^i_0$. Moreover, by our assumptions on  $(\zeta_1^i, \nu^i_1, R^i_1)$, the processes $(Y_n^1, R_n^1 + d \nu_n^1)_{n \geq 0}$ and
$(Y_n^2, R_n^2)_{n \geq 0}$ are look-around random walks, which satisfy the conditions in Lemma~\ref{lem:7}. Therefore, by the lemma,  if $\rho > 0$ is chosen large enough, we have
\begin{equation}
\label{e:121.1}
\bbP (\sigma^i > u) \geq C / \sqrt{u} 
\ , \ \ u \geq 0 \,,\, i=1,2 \,.
\end{equation}
This handles the first probability in~\eqref{e:141}. For the second, we write
\begin{equation}
\label{e:124}
\bbP \big( \sigma^2 > n \,,\, N^2_{[x,y]} > n \big) = 
\bbP \big( \sigma^2 > n ) - \bbP \big( \sigma^2 > n \,,\, N^2_{[x,y]} \leq n \big) 
\end{equation}
By~\eqref{e:121.1}, the first term is lower bounded by $C/\sqrt{n}$ once $\rho$ is chosen large enough. We wish to show now that the second term can be upper bounded by $C/(2\sqrt{n})$ be choosing then $y$ small enough.

Thanks to Lemma~\ref{lem:50}, writing $\mu$ for $\bbE \zeta^2_1$, we have for all $k \geq 0$ and $y < 0$, 
\begin{equation}
\label{e:125}
\begin{split}
\bbP (S^2_k - R^2_{k+1} \leq y) \leq &  \bbP (S^2_k \leq k \mu/2 + y/2 ) + 
\bbP (R^2_{k+1} \geq k \mu/2 - y/2 ) \\
& \leq C \rme^{-C' (k-y)} \,. 
\end{split}
\end{equation}
Therefore,
\[
\bbP \big( N^2_{[x,y]} > n^{1/16} \big) 
\leq \sum_{k=n^{1/16}}^\infty \bbP ( S^2_k - R^2_{k+1} \leq y ) \leq \rme^{-C n^{1/16}} \,.
\]
At the same time,
\[
\begin{split}
\bbP \big( \exists k \leq n^{1/16} \,:\, |S^2_k|+|T^2_k|+R^2_{k+1} > n^{1/4} \big)
& \leq  \bbP \big( \exists k \leq n^{1/16}+1 \,:\, |\zeta^2_k|+ \nu^2_k + R_k^2 > n^{1/16} \big) \\
& \leq (n^{1/16}+1) \rme^{-C n^{1/16} } \leq \rme^{-C' n^{1/16}} \,.
\end{split}
\]
It follows that the second term on the right hand side of~\eqref{e:124} is bounded above by
\begin{equation}
\label{e:128}
\rme^{-C n^{-1/{16}}} + 
\sum_{k=1}^{n^{1/{16}}} \sum_{w = -2n^{1/4}}^y \sum_{m=k}^{n^{1/4}}
	\bbP \big(S^2_k - R^2_{k+1} = w \,,\, T^2_k = m \big) 
		\bbP \big(\sigma^2 > n-k \,\big|\, Y^2_0 = w - d m \big) \,,
\end{equation}
where conditioning in the second probability is only formal and means that $(Y^2_n)_{n \geq 0}$ is redefined so that $Y_0^2 = w-dm$.

In the range of the sums, we can use Lemma~\ref{lem:7.5} to upper bound the second probability by
\[
C\frac{ z - w + d m}{\sqrt{n-k}} \leq C'\frac{z - w + d m}{\sqrt{n}} \,.
\]
Plugging this into~\eqref{e:128} and removing some of the restrictions on the sums, we get as an upper bound,
\begin{equation}
\label{e:210}
\rme^{-C n^{-1/16}} + C' n^{-1/2} \sum_{k=1}^\infty \bbE \Big( \big(z - S^2_k + R^2_{k+1} + d T^2_k \big) \,1_{\{S^2_k - R^2_{k+1} \leq y\}} \Big) \,.
\end{equation}
Using Cauchy-Schwartz and~\eqref{e:125}, the last expectation can be further bounded above by
\[
C \big( z^2 + \bbE |S^2_k|^2 + \bbE |R^2_{k+1}|^2 + \bbE |T^2_k|^2 \big)^{1/2}
\big( \bbP \big(S^2_k - R^2_{k+1} \leq y \big)  \big)^{1/2}
\leq C (z + k) \rme^{-C' (k-y)} \,.
\]
Consequently,~\eqref{e:210} is bounded above by
\[
\rme^{-C n^{-1/16}} + C' n^{-1/2} (z+1) \rme^{C''y} \sum_{k=1}^\infty k \rme^{-C''' k} 
\leq \rme^{-C n^{-1/16}} + C' n^{-1/2} (z+1) \rme^{C''y}  \,.
\]

Choosing first $\rho$ (and hence $z$) large enough and then $y$ small enough, by~\eqref{e:121.1} for $i=2$ and~\eqref{e:124}, we can indeed ensure that
\[
\bbP \big( \sigma^2 > n \,,\, N^2_{[x,y]} > n \big) \geq C/\sqrt{n}  \,.
\]
Together with~\eqref{e:121.1} for $i=1$ we get~\eqref{e:141} and hence also~\eqref{e:21}.
\end{proof}
\LongVersionEnd 

\section{One Scout on $\bbZ^{1}$}
\label{section:one-scout}
The arguments leading to the proof of Theorem~\ref{thm:A} can be adapted to
yield a proof for the case
$d = 1$.
\begin{theorem}
\label{thm:B}
Let
$(X_n, Q_n)_{n \geq 0}$
be a single scout process on $\bbZ$ with protocol
$\cP = \langle 1, \cS, 0, q_0, \Pi \rangle$.
Then there exists some grid point
$x \in \bbZ$
of which the expected hitting time is infinite, namely,
\begin{equation} 
\label{e:B.1}
\bbE \left(
\inf \{ n \geq 0 \, : \, X_n = x \}
\right)
\, = \,
\infty \, .
\end{equation}
\end{theorem}
\begin{proof}
Recalling that the automaton process $(Q_n)_{n \geq 0}$ is a Markov chain on $\cS$, we let $\cS_1, \dots, \cS_l$ for $l \geq 1$, be its irreducible recurrent state classes. Suppose first that $q_0 \in \cS_m$ for some $m \in \{1, \dots, l\}$, let $T_0 = 0$ and for $k=1, \dots$, set
\begin{equation}
\label{e:B.67}
\begin{split}
T_k &:= \inf \{n > T_{k-1} :\: Q_n = q_0 \} \,, \\
\zeta_k &:= X_{T_k} -  X_{T_{k-1}} \,, \\
\nu_k &:= T_k-T_{k-1} \, . \\
\end{split}
\end{equation}
By the Markov property and spatial homogeneity of the process $(X_n, Q_n)_{n \geq 0}$, the 
pairs $(\zeta_k, \nu_k)_{k \geq 1}$ are i.i.d. Moreover, standard Markov chain theory shows that
\begin{equation}
\label{e:B.2}
\bbP ( |\zeta_k| + \nu_k > u ) \leq C^{-1} \rme^{-C u} \,,
\end{equation}
for some $C > 0$. In particular, if we set for $n \geq 0$,
\begin{equation}
S_n := x_0 + \sum_{k=1}^n \zeta_k = X_{T_n} \ , \ \ 
\end{equation}
then $(S_n)_{n \geq 0}$ is a random walk on $\bbZ$. If $\bbP(\zeta_1 = 0) = 1$, then we appeal to Lemma~\ref{lem:12} to obtain $r > 0$, such that $\bbP(|X_{T_1}| <  r) = 1$ and accordingly we set $R_n := r$ for all $n \geq 1$. Otherwise, we set $R_n := \nu_n$. In both cases we then have
\begin{equation}
\inf \{ n \geq 0 \, : \, X_n = x \}
\geq
\inf \{ n \geq 0 \, : \, |S_n - x| \leq R_{n+1} \} \,.
\end{equation}

\ShortVersion 
We call the process $(S_n, R_n)_{n \geq 0}$ a \emph{look-around random walk}
(refer to the full version for further details).
\ShortVersionEnd 
\LongVersion 
Now, in light of~\eqref{e:B.2}, the process $(S_n, R_n)_{n \geq 0}$ is a
look-around random walk of the type treated in Subsection~\ref{sub:7.1},
starting from $0$.
\LongVersionEnd 
Moreover, by definition if $\bbP(\zeta_1 = 0) = 1$, then
$R_n$ is bounded by a deterministic quantity.
By employing Lemma~\ref{lem:6} or Lemma~\ref{lem:18},
we conclude that there exist infinitely many
$x \in \bbZ$
such that 
\begin{equation}
\bbE \big( \inf \{ n \geq 0 \, : \, X_n = x \} \big)
\geq
\bbE \big(  \inf \{ n \geq 0 \, : \, |S_n - x| \leq R_{n+1} \} \big) = \infty \,.
\end{equation}
In particular, this shows~\eqref{e:B.1}.

If $q_0$ does not belong to any of $\cS_m$ for $m \in \{1, \dots, l\}$, then
we set
\begin{equation}
\sigma := \inf \big\{ n:\: Q_n \in \cup_{m=1}^l \cS_m \big\}
\end{equation}
and let $L$ be such that $Q_\sigma \in \cS_L$. Standard Markov chain theory gives that $\sigma < \infty$ almost surely and hence $L$ is well defined. Since $\sigma$ is a stopping time, conditional on $\cF_\sigma$, the process $(X_{\sigma+n}, Q_{\sigma+n})_{n \geq 0}$ is distributed as the original process starting from $x_0 = X_\sigma$ and $q_0 = Q_\sigma$. Moreover, since $\sigma$ is finite, the number of visited grid points $x \in \bbZ$ up to this time is finite. Repeating the argument above with $m = L$, there is at least one (random, $\cF_\sigma$-measurable) $x' \in \bbZ$ which was not visited up to time $\sigma$ and such that
\begin{equation}
\bbE \big(\inf \{ n \geq 0 \, : \, X_{\sigma + n} = x' \} \,\big|\, \cF_\sigma \big) = \infty \,.
\end{equation}
This readily implies~\eqref{e:B.1} for some (deterministic) $x \in \bbZ$.
\end{proof}

\section*{Acknowledgments}
The authors would like to thank Roger Wattenhofer and Tobias Langner for helpful discussions.
The work of O.L. was supported in part by the European Union's - Seventh Framework Program (FP7/2007-2013) under grant agreement no 276923 -- M-MOTIPROX. The work of L.C. was supported by a Technion MSc. fellowship.

\clearpage
\pagenumbering{gobble}

\bibliographystyle{abbrv}
\bibliography{references}

\begin{thebibliography}{10}

\bibitem{Albers2000}
S.~Albers and M.~Henzinger.
\newblock Exploring unknown environments.
\newblock {\em {SIAM Journal on Computing}}, 29:1164--1188, 2000.

\bibitem{Aleliunas1979}
R.~Aleliunas, R.~M. Karp, R.~J. Lipton, L.~Lovasz, and C.~Rackoff.
\newblock Random walks, universal traversal sequences, and the complexity of
  maze problems.
\newblock In {\em {Proceedings of the 20th Annual Symposium on Foundations of
  Computer Science (SFCS)}}, pages 218--223, 1979.

\bibitem{alio1999wiener}
L.~Alili and R.~Doney.
\newblock Wiener--hopf factorization revisited and some applications.
\newblock {\em Stochastics: An International Journal of Probability and
  Stochastic Processes}, 66(1-2):87--102, 1999.

\bibitem{Alon11}
N.~Alon, C.~Avin, M.~Kouck{\'{y}}, G.~Kozma, Z.~Lotker, and M.~R. Tuttle.
\newblock Many random walks are faster than one.
\newblock {\em Combinatorics, Probability {\&} Computing}, 20(4):481--502,
  2011.

\bibitem{averbakh1996heuristic}
I.~Averbakh and O.~Berman.
\newblock A heuristic with worst-case analysis for minimax routing of two
  travelling salesmen on a tree.
\newblock {\em Discrete Applied Mathematics}, 68(1):17--32, 1996.

\bibitem{Awerbuch1999}
B.~Awerbuch and M.~Betke.
\newblock Piecemeal graph exploration by a mobile robot.
\newblock {\em Information and Computation}, 1999.

\bibitem{baezayates1993searching}
R.~A. Baezayates, J.~C. Culberson, and G.~J. Rawlins.
\newblock Searching in the plane.
\newblock {\em Information and computation}, 106(2):234--252, 1993.

\bibitem{Bender1994}
M.~Bender and D.~Slonim.
\newblock The power of team exploration: Two robots can learn unlabeled
  directed graphs.
\newblock In {\em Proceedings of Foundations of Computer Science (FOCS)}, pages
  75--85, Nov 1994.

\bibitem{Blum1978}
M.~Blum and D.~Kozen.
\newblock On the power of the compass (or, why mazes are easier to search than
  graphs).
\newblock In {\em Proceedings of the 19th Annual Symposium on Foundations of
  Computer Science (FOCS)}, pages 132--142, 1978.

\bibitem{Blum1977}
M.~Blum and W.~J. Sakoda.
\newblock On the capability of finite automata in 2 and 3 dimensional space.
\newblock In {\em Proceedings of the 18th Annual Symposium on Foundations of
  Computer Science (FOCS)}, pages 147--161, 1977.

\bibitem{Budach1978}
L.~Budach.
\newblock Automata and labyrinths.
\newblock {\em Mathematische Nachrichten}, pages 195--282, 1978.

\bibitem{chrobak2015group}
M.~Chrobak, L.~Gasieniec, T.~Gorry, and R.~Martin.
\newblock Group search on the line.
\newblock In {\em International Conference on Current Trends in Theory and
  Practice of Informatics}, pages 164--176. Springer, 2015.

\bibitem{cooper2009multiple}
C.~Cooper, A.~Frieze, and T.~Radzik.
\newblock Multiple random walks in random regular graphs.
\newblock {\em SIAM Journal on Discrete Mathematics}, 23(4):1738--1761, 2009.

\bibitem{Deng1999}
X.~Deng and C.~Papadimitriou.
\newblock Exploring an unknown graph.
\newblock {\em Journal of Graph Theory}, 32:265--297, 1999.

\bibitem{Diks2004}
K.~Diks, P.~Fraigniaud, E.~Kranakis, and A.~Pelc.
\newblock Tree exploration with little memory.
\newblock {\em Journal of Algorithms}, 51:38--63, 2004.

\bibitem{Doepp1971}
K.~D{\"{o}}pp.
\newblock Automaten in labyrinthen.
\newblock {\em Elektronische Informationsverarbeitung und Kybernetik},
  7(2):79--94, 1971.

\bibitem{Duncan2006}
C.~A. Duncan, S.~G. Kobourov, and V.~S.~A. Kumar.
\newblock Optimal constrained graph exploration.
\newblock {\em ACM Transactions on Algorithms (TALG)}, 2(3):380--402, 2006.

\bibitem{EmekLSUW2015}
Y.~Emek, T.~Langner, D.~Stolz, J.~Uitto, and R.~Wattenhofer.
\newblock How many ants does it take to find the food?
\newblock {\em Theor. Comput. Sci.}, 608:255--267, 2015.

\bibitem{EmekLUW2014}
Y.~Emek, T.~Langner, J.~Uitto, and R.~Wattenhofer.
\newblock Solving the {ANTS} problem with asynchronous finite state machines.
\newblock In {\em Distributed Computing the 41st International Colloquium on
  Automata, Languages, and Programming {ICALP}}, pages 471--482, 2014.

\bibitem{FeinermanK2012}
O.~Feinerman and A.~Korman.
\newblock Memory lower bounds for randomized collaborative search and
  implications for biology.
\newblock In {\em Proceedings of the 26th International Symposium on
  Distributed Computing {DISC}}, pages 61--75, 2012.

\bibitem{FeinermanKLS2012}
O.~Feinerman, A.~Korman, Z.~Lotker, and J.~Sereni.
\newblock Collaborative search on the plane without communication.
\newblock In {\em Proceedings of {ACM} Symposium on Principles of Distributed
  Computing, {PODC}}, pages 77--86, 2012.

\bibitem{Fomin2008}
F.~V. Fomin and D.~M. Thilikos.
\newblock An annotated bibliography on guaranteed graph searching.
\newblock {\em Theoretical Computer Science}, 399(3):236--245, 2008.

\bibitem{fraigniaud2006collective}
P.~Fraigniaud, L.~Gasieniec, D.~R. Kowalski, and A.~Pelc.
\newblock Collective tree exploration.
\newblock {\em Networks}, 48(3):166--177, 2006.

\bibitem{Fraigniaud2004}
P.~Fraigniaud and D.~Ilcinkas.
\newblock Digraphs exploration with little memory.
\newblock In {\em Proceedings of the 21st Symposium on Theoretical Aspects of
  Computer Science (STACS)}, pages 246--257, 2004.

\bibitem{Fraigniaud2005}
P.~Fraigniaud, D.~Ilcinkas, G.~Peer, A.~Pelc, and D.~Peleg.
\newblock Graph exploration by a finite automaton.
\newblock {\em Theoretical Computer Science}, 345(2--3):331--344, 2005.
\newblock Mathematical Foundations of Computer Science 2004Mathematical
  Foundations of Computer Science 2004.

\bibitem{LangnerUSW2014}
T.~Langner, J.~Uitto, D.~Stolz, and R.~Wattenhofer.
\newblock Fault-tolerant {ANTS}.
\newblock In {\em Prooceedings of the 28th International Symposium on
  Distributed Computing {DISC}}, pages 31--45, 2014.

\bibitem{LawlerL2010}
G.~F. Lawler and V.~Limic.
\newblock {\em Random Walk: A Modern Introduction}.
\newblock Cambridge University Press, New York, NY, USA, 2010.

\bibitem{lopez2001parallel}
A.~L{\'{o}}pez{-}Ortiz and G.~Sweet.
\newblock Parallel searching on a lattice.
\newblock In {\em Proceedings of the 13th Canadian Conference on Computational
  Geometry}, pages 125--128, 2001.

\bibitem{Panaite1998}
P.~Panaite and A.~Pelc.
\newblock Exploring unknown undirected graphs.
\newblock In {\em Proceedings of the 9th Annual ACM-SIAM Symposium on Discrete
  Algorithms (SODA)}, pages 316--322, 1998.

\bibitem{Rollik1979}
H.~Rollik.
\newblock Automaten in planaren graphen.
\newblock In K.~Weihrauch, editor, {\em Theoretical Computer Science},
  volume~67 of {\em Lecture Notes in Computer Science}, pages 266--275.
  Springer Berlin Heidelberg, 1979.

\end{thebibliography}

\end{document}